\documentclass[10pt]{amsart}
\usepackage{amssymb,latexsym,amsthm}
\usepackage[bookmarksopen=true]{hyperref}
\usepackage[dvips]{graphicx}
\usepackage{color}

\newcommand{\R}{\mathbb R}
\newcommand{\N}{\mathbb N}
\newcommand{\bea}{\begin{eqnarray}}
\newcommand{\eea}{\end{eqnarray}}
\newcommand{\norm}[1]{\left\Vert#1\right\Vert}
\newcommand{\D}{\displaystyle}

\newcommand{\beq}{\begin{equation}}
\newcommand{\eeq}{\end{equation}}
\newtheorem{theorem}{Theorem}[section]

\newtheorem{remark}[theorem]{Remark}
\newtheorem{lemma}[theorem]{Lemma}

\newtheorem{definition}[theorem]{Definition}

\begin{document}

\title[Determination of the Calcium channel distribution....]{Determination of the Calcium channel distribution in the olfactory system} 

\author{C. Conca$^1$, R.  Lecaros$^{1,2}$, J. H. Ortega$^1$ \& L. Rosier$^3$}
\address{$^1$Centro de Modelamiento Matem\'atico (CMM) and Departamento de Ingenier\'ia Matem\'atica,
Universidad de Chile (UMI CNRS 2807), Avenida Blanco Encalada 2120, Casilla 170-3, Correo 3, Santiago, Chile.}
\email{cconca@dim.uchile.cl,jortega@dim.uchile.cl}

\smallskip\

\address{$^2$Basque Center for Applied Mathematics - BCAM, Mazarredo 14, E-48009, Bilbao, Basque Country, Spain.}
\email{rlecaros@bcamath.org}

\smallskip\

\address{$^3$Institut Elie Cartan, UMR 7502 UdL/CNRS/INRIA,
B.P. 70239, 54506 Vand\oe uvre-l\`es-Nancy Cedex, France.}
\email{Lionel.Rosier@univ-lorraine.fr}

\begin{abstract}
In this paper we study a linear inverse problem with a biological interpretation, which is
modeled by a Fredholm integral equation of the first kind.
When the kernel in the Fredholm equation is represented by step functions, we obtain identifiability, stability 
and reconstruction  results.
Furthermore, we provide a  numerical reconstruction algorithm for the kernel, whose main feature is that a non-regular mesh
has to be used to ensure the 
invertibility of the matrix representing the numerical discretization of the system. 
Finally, a second identifiability result for a polynomial approximation of degree less than nine of the kernel is 
also established. 
\end{abstract}
\maketitle

\section{Introduction}
\
\par

In this work we study an integral inverse problem coming from the biology of the olfactory system. 
The transduction of an odor into an electrical signal is accomplished by a depolarising influx 
of ions through cyclic-nucleotide-gated (CNG) channels in the membrane. Those channels, 
that form the lateral surface of the cilium, are activated by adenosine 3', 5'-cyclic 
monophosphate (cAMP).
 
D.A.~French et al. \cite{Bio:FFGKK} propose a mathematical model for the dynamics of 
cAMP concentration, consisting of two nonlinear differential equations and a constrained 
Fredholm integral equation of first kind. The unknowns of the problem are the concentration 
of cAMP, the membrane potential 
and the distribution $\rho$  of CNG channels 
along the length of a cilium. 
A very natural issue is whether it is possible to recover the distribution of CNG channels 
along the length of a cilium by only measuring  the electrical activity produced by the 
diffusion of cAMP into cilia. A simple numerical method  to obtain estimates of channels 
distribution is also proposed in  \cite{Bio:FFGKK}. Certain computations indicate that this 
mathematical problem is ill-conditioned.

Later, D.A.~French \& D.A.~Edwards \cite{Bio:FrenchEdwards} studied the above inverse 
problem by using perturbation techniques. A simple perturbation approximation was derived 
and used to solve the inverse problem, and to obtain estimates of the spatial distribution of 
CNG ion channels. A one-dimensional computer minimization and a special delay iteration 
were used with the perturbation formulas to obtain approximate channel distributions in the 
cases of simulated and experimental data. Moreover, D.A.~French \& C.W.~Groetsch \cite{Bio:FrenchGroetsch} introduced some  simplifications and approximations  in the problem, 
obtaining an analytical solution for the inverse problem. A numerical procedure was proposed 
for a class of integral equations suggested by this simplified model and numerical results 
were compared to laboratory data.

In this paper we consider the linear problem proposed in \cite{Bio:FrenchGroetsch}, with 
an improved approximation of the kernel, along with studying the identifiability, stability and 
numerical reconstruction for the corresponding inverse problem.

The inverse problem consists in determining a function $\rho = \rho(x)>0$ from the 
measurement of  
    \beq\label{Int:Emodel1:01}
    I_m[\rho](t)=J_0\int_0^L\rho(x)K_m(t,x)dx,
    \eeq
for $t \in I,$ where $I$ is a time interval, $\rho$ is the channel distribution, $J_0$ is a 
positive constant and the kernel $K_m(t,x)$ is defined by   
    \beq\label{Int:B_DkerAprox}
    K_m(t,x)=F_m(w(t,x)),
    \eeq 
where  $w(t,x)$, defined in \eqref{B_H01},  represents an approximation of the  concentration of cAMP at point $(t,x)$ and 
$F_m$ is a step function approximation of the Hill function $F$, given by
\beq\label{Int:Dhillfunc}
F(x)=\frac{x^n}{x^n+K_{1/2}^n}.
\eeq
Here, the exponent $n$ is an experimentally determined parameter and  $K_{1/2}>0$ is a constant which corresponds to the  half-bulk concentration. 

Under a strong assumption about the regularity of $\rho$ (namely, $\rho $ is analytic), 
we obtain in Theorem~\ref{B_Tiden01:01} an identifiability result for \eqref{Int:Emodel1:01}
with a single measurement of $I_m[\rho ]$ on an {\em arbitrary small} interval around zero. 
The second identifiability result, Theorem~\ref{B_Tiden02:02}, requires weaker regularity 
assumptions about $\rho$ (namely, $\rho \in L^2(0,L)$), but it requires the measurement of $I_m[\rho ]$  on a large time interval. 

Furthermore,  in Theorem~\ref{B_Tstab01}, using appropriate weighted norms and the Mellin transform, 
we obtain a general stability result for the operator $I_m[\rho]$ for $\rho\in L^2(0,L)$. Using a non-regular mesh for the approximation of $F_m$, 
we develop a reconstruction procedure in Theorem ~\ref{B_Trec} 
to recover $\rho$ from $I_m$.  Additionally, for this non-regular mesh, 
a general stability result for a large class of norms is rigorously established
 in Theorem~\ref{B_Tstab03}.

\smallskip

On the other hand, we also investigate the same inverse problem with another approximation of the kernel obtained by replacing 
Hill's function by its Taylor expansion of degree $m$ around $c_0>0$. 

More precisely, the polynomial kernel 
approximation is defined as
\beq\label{IntB_DkerAprox02}
PK_m(t,x)=P_m(c(t,x)-c_0),
\eeq
where $P\in \R [x]$, deg$\, (P)\le m$ is such that 
\[
F(x)= P_m(x-c_0) + O (|x-c_0|^{m+1}), 
\]
 and  $c(t,x)$, the concentration of cAMP,  is defined as the solution of the diffusion problem \eqref{Int:Emodel2}.
Thus, the total current with polynomial approximation is given by 
\beq\label{Int:B_DpoliCurrent}
PI_m[\rho](t)=\int_0^L\rho(x)PK_m(t,x)dx \quad\forall t>0.
\eeq

In Theorem \ref{B_Tiden03} we derive an identifiability result for the operator $PI_m$, 
when the degree of $P_m$ is less than nine.   

\smallskip
The paper is organized as follows. In Section 2, we set the problem, introduce the principal 
assumptions and some operator $\Phi_m$ that we use to derive the main 
results regarding the operator $I_m$. These results are presented in Section~3. Section~4 is 
devoted to prove the identifiability theorems. Section~5 is devoted to the proof of 
Theorem~\ref{B_Tstab01} concerning the stability of $I_m$. The proof of  the results 
involving the reconstruction procedure are developed in Section~6, while the numerical 
algorithm and examples are shown in Section~7. Finally, in Section~8, we prove an 
identifiability result for  $PI_m$ (Theorem~\ref{B_Tiden03}).  

\section{Setting the problem}
\
\par

In this section we will set the mathematical model related to the inverse problem 
arising in olfaction experimentation.

\smallskip
The starting point is the linear model introduced in \cite{Bio:FrenchGroetsch}.
As already mentioned, a nonlinear integral equation model was developed 
in \cite{Bio:FFGKK} to determine the spatial distribution of ion channels along the length 
of frog olfactory cilia. The essential nonlinearity in the model arises from the binding of the 
channel activating ligand to the cyclic-nucleotide-gated ion channels as the ligand 
diffuses along the length of the cilium. We investigate a linear model for this process,
in which the binding mechanism is neglected, leading to a particular type of linear 
Fredholm integral equation of the first kind with a diffusive kernel. 
The linear inverse problem consists in determining $\rho = \rho(x)>0$ from the measurement of 
    \beq\label{Int:Emodel1}
    I[\rho](t)=J_0\int_0^L\rho(x)K(t,x)dx,
    \eeq
    where the kernel is   
    \beq\label{ABC}
    K(t,x)=F(c(t,x)),
    \eeq 
    $F$ being given by \eqref{Int:Dhillfunc}
and  $c$ denoting the concentration of cAMP which is governed by  the following diffusion
boundary value problem:
\beq\label{Int:Emodel2}
\left\{
\begin{array}{rcll}
\D\frac{\partial c}{\partial t}-D\frac{\partial^2 c}{\partial x^2}&=&0,& t> 0,\; x\in(0,L),\\
c(0,x)&=&0,& x\in(0,L),\\
c(t,0)&=&c_0,& t>0,\\
\D\frac{\partial c}{\partial x}(t,L)&=&0,& t>0.
\end{array}\right.
\eeq

The (unknown) function $\rho$ is the ion channel density function, and
$c$ is the concentration of a channel activating ligand that is diffusing from 
left-to-right in a thin cylinder (the interior of the cilium) of length $L$ with 
diffusivity constant $D$. $I[\rho](t)$ is a given total transmembrane current, 
the constant $J_0$ has units of current/length, and $c_0$ is the maintained 
concentration of cAMP at the open end of the cylinder (while $x = L$ is considered as
the closed end).

\smallskip

Thus, the inverse problem consists in obtaining $\rho$ from the 
measurement  of $I[\rho](t)$  in some time interval. 

\smallskip
We note that this is a Fredholm integral equation of the first kind; that is,
\beq
I[\rho](t)=\int_0^LK(t,x)\rho(x)dx,
\eeq
where $K(t,x)=J_0F(c(t,x))$ is the {\em kernel} of the operator.  
The associated inverse problem is  often ill-posed. For example, if $K$ is 
sufficiently smooth, the operator defined above is compact from $L^p(0,L)$ to $L^p(0,T)$ for  $1<p<\infty$.  
Even if the operator $I$ is injective, its inverse will not be continuous. Indeed, if $I$ is compact and $I^{-1}$ is continuous, 
 then it follows that the identity map in $L^p(0,L)$ is compact, a property  which is clearly false.
 
In what follows, we consider a simplified version of the above problem under 
more general assumptions than those in \cite{Bio:FrenchGroetsch}. 
Let us introduce the following generic assumptions:

\begin{itemize}
\item[(i)] We can approximate the solution $c(t,x)$ of \eqref{Int:Emodel2}
 as follows:
\beq
c(t,x)\simeq w(t,x)= c_0\, \textrm{erfc}\left(\frac {x}{2\sqrt{Dt}}\right),\label{B_H01}
\eeq
  where $\textrm{erfc}$ is the complementary error function:
\beq\label{B_H02}
\textrm{erfc}(z)= 1-\frac{2}{\sqrt{\pi}}\int\limits_0^z \textrm{exp}(-\tau^2)d\tau.
\eeq

\item[(ii)]  We consider the following approximation of 
Hill's function given in  \eqref{Int:Dhillfunc}
\beq \label{B_H03}
F(x)\simeq F_m(x)=F(c_0)\sum_{j=1}^m a_jH(x-\alpha_j)\quad\forall x\in[0,c_0],\eeq
where $H$ is the Heaviside unit step function, i.e.
\beq \label{B_H04}
H(u)=\left\{
\begin{array}{ccc}
1 & \textrm{if } u\geq 0, \\ \\
0 & \textrm{if } u< 0,
\end{array}
\right.
\eeq
 and  $a_j,\alpha_j$ are positive constants such that
\beq\label{B_H05}
\sum_{j=1}^ma_j=1,
\eeq
and 
\beq\label{B_H06}
0<\alpha_1<\alpha_2<\cdot\cdot\cdot<\alpha_m<c_0,
\eeq
and hence, $\{\alpha_j\}_{j=1}^m$ defines a partition of the interval $(0,c_0).$ 
\end{itemize}

With the above assumptions we define  the approximate total current 
\beq\label{B_Dfunc01}
I_m[\rho](t)=J_0\int_0^L\rho(x)K_m(t,x)dx,
\eeq
where 
\beq\label{B_DkerAprox}
K_m(t,x)=F_m(w(t,x))=F_m\left(c_0\,\textrm{erfc}\big( \frac {x}{2\sqrt{Dt}}\big)\right).
\eeq

Therefore, our inverse problem consists in  recovering $\rho$ from the measurement 
of $I_m[\rho](t)$ for all $t\geq 0.$

\smallskip
For any $\gamma>0$, we consider  the function $\sigma_\gamma(x)=|x|^\gamma$,   
and introduce the following weighted norms 
$$
\begin{array}{lcl}
\D\norm{f}_{0,\gamma,b}&=&\D\norm{\sigma_\gamma f}_{L^2(0,b)}
,\\ \\
\D\norm{f}_{1,\gamma,b} &=&\D\norm{\sigma_\gamma f}_{H^1(0,b)},\\ \\
\D\norm{f}_{-1,\gamma,b} &=&\D\norm{\sigma_\gamma f}_{H^{-1}(0,b)}.
\end{array}
$$
We set 
  \beq 
   L_k=L/{\beta_k}\;\; \textrm{ for } k=1,...,m,\;\;\; \textrm{and  } L_0=0,
\eeq
where 
 \beq\label{B_Dbeta01}
 \beta_j=\textrm{erfc}^{-1}(\alpha_j/c_0)2\sqrt{D}\;\;\textrm{for } j=1,...,m. 
 \eeq
 
On the other hand,  we have 
\begin{equation}
\begin{array}{rll}
I_m[\rho](t) & = &\displaystyle J_0\int_0^L\rho(x)K_m(t,x)dx \\
&=&\displaystyle  J_0F(c_0)\sum_{j=1}^m a_j \int_0^L \rho(x) H(w(t,x)-\alpha_j)dx \\ 
&=&\displaystyle  J_0F(c_0)\sum_{j=1}^m a_j \int\limits_{ G_j(t)\cap(0,L)} \rho(x)dx,
\end{array}
\end{equation}
with $G_j(t)=\{x\in\R:\;\;w(t,x)\geq \alpha_j\}$. 
Since the ``$\textrm{erfc}$'' function is decreasing, we have 
\beq
G_j(t)=\big[0,\beta_j\sqrt{t}\big],\eeq
where  $\{\beta_j\}_{j=1}^m$ are given by \eqref{B_Dbeta01}. (Note that 
$\beta_1>\beta_2>\cdot\cdot\cdot>\beta_m$.) Thus, we have 
\beq\label{B_D:ec_Im}
I_m[\rho](t)=J_0F(c_0)\bigg(\sum_{j=1}^m a_j \int\limits_0^{h_j(\sqrt{t})} \rho(x)dx\bigg),
\eeq
where  $h_j(s)=\min\{L,\beta_j s\}.$

\smallskip

Next, we define
\beq\label{B_L04E01}
\begin{array}{ccc}
\Phi_m[\varphi](t)&=&\D\sum_{j=1}^m a_j \varphi\left(h_j(t)\right)\;\; \forall t\geq 0,
\end{array} 
\eeq
and obtain 
 \beq\label{rela:Phi:Im}
  I_m[\rho](t)=J_0F(c_0)\Phi_m[\varphi](\sqrt{t}),
 \eeq
 with 
 $$
 \varphi(x)=\int_0^x\rho(\tau)d\tau.
 $$
 
Clearly,  $\Phi_m$ is linear, and it follows from \eqref{B_H05} that $\Phi _m (1)=1$,
and that for any $f \in L^\infty(0,L)$ we have
$$
\norm{\Phi_m[f]}_{ L^\infty(0,L_m)}\leq \norm{f}_{ L^\infty(0,L)}.
$$
Furthermore, for any $f\in C([0,L])$ with $f(L)=0$, we have 
\beq\label{cont:PhiLp}
\norm{\Phi_m[f]}_{ L^p(0,L_m)}\leq
\left(\sum_{j=1}^m a_j\beta_j^{-1/p}\right)\norm{f}_{ L^p(0,L)},\qquad 1\le p<\infty.
\eeq

\section{Main results}
\
\par

In this section we present the main results in this paper.  We begin by  studying 
the functional $\Phi_m$, defined in \eqref{B_L04E01}. It is worth noticing with \eqref{rela:Phi:Im} that the 
identifiability for $\Phi_m$ is equivalent to the identifiability for $I_m$. 

\smallskip
Firstly, we discuss some identifiability results for the operator $\Phi_m$.  
We begin with the analytic case.

\smallskip
\begin{theorem}[Identifiability for analytic functions] \label{B_Tiden01}
If $\varphi:[0,L]\to \R$ is an analytic function satisfying
\beq
\Phi_m[\varphi](t)=0\quad\forall t\in(0,\delta)
\eeq
for some $\delta>0$, then $\varphi\equiv 0$ in $[0,L].$
\end{theorem}

\smallskip
The second identifiability  result  requires less regularity for $\varphi$, provided that a measurement on a sufficiently large
time interval is available.

\begin{theorem}\label{B_Tiden02} Let $\varphi:[0,L]\to \R$ be a given function 
satisfying
\beq
\Phi_m[\varphi](t)= 0 \quad\forall t\in[0,L_m].
\eeq
Then $\varphi\equiv 0$ in $[0,L].$ 
\end{theorem}
 
The proof of  Theorem \ref{B_Tiden02} uses algebraic arguments and  it gives us an
idea on how the kernel could be reconstructed and also how one can envision 
a numerical algorithm. 
The corresponding identifiability results for the operator $I_m$ are as follows.
\begin{theorem}[Identifiability for analytic functions] \label{B_Tiden01:01}
If $\rho:[0,L]\to \R$ is an analytic function such that 
\beq
I_m[\rho](t)=0 \quad\forall t\in(0,\delta),
\eeq
for some $\delta>0$, then $\rho\equiv 0$ in $[0,L].$
\end{theorem}

\begin{theorem}\label{B_Tiden02:02} Let $\rho:[0,L]\to \R$ be a given function in 
$L^2(0,L)$ such that  
\beq
I_m[\rho](t)=0 \quad\forall t\in[0,L_m^2].
\eeq
Then $\rho\equiv 0$ in $[0,L].$ 
\end{theorem}

Theorems~\ref{B_Tiden01:01} and  \ref{B_Tiden02:02} follow at once 
from  Theorems~\ref{B_Tiden01} and  \ref{B_Tiden02} by letting 
\[
\varphi (x) = \int_0^x \rho (\tau ) d\tau .
\] 
Let us now proceed to the continuity and stability results.

\begin{theorem} \label{B_Tcont02} Let $\varphi \in H^1(0,L)$ be a given function.  
Then there exists  a constant $\tilde C_1>0$ such that   
\beq
 \norm{\Phi_m[\varphi]}_{H^1(0,L_m)}\leq \tilde C_1 \norm{\varphi}_{H^1(0,L)},
\eeq
where $\tilde C_1$  depends only on $L,\beta_1 $ and $\beta_{m}$.
\end{theorem}

We are now in a position to state our first main result. Firstly, we define the function 
\beq\label{B_DCons01}
\Lambda^\gamma_m(s)=\left| \sum_{j=1}^ma_j\beta_j^{-(\frac{1}{2}+\gamma-is)}\right|,
\eeq
where $i=\sqrt{-1}$ is the imaginary unit.

\begin{theorem}\label{B_Tstab02}
Let  
$\varphi \in C([0,L])$  be a given function. Then there exists a constant 
$\gamma_0\in\R$ such that for any $\gamma>\gamma_0$,   
\begin{equation}
C_\gamma\norm{\varphi(\cdot)-\varphi(L)}_{0,\gamma,L}\leq 
\norm{\Phi_m[\varphi](\cdot)-\Phi_m[\varphi](L_m)}_{0,\gamma,L_m},
\label{A1}
\end{equation}
where 
$$
C_\gamma:=\inf_{s\in\R}\Lambda^\gamma_m(s)>0.
$$ 
\end{theorem}

It is worth noting that \eqref{A1} can be viewed as an inverse inequality of  \eqref{cont:PhiLp} for $p=2$ and for functions $\varphi \in
\{ f\in C( [ 0,L ]); \  f(L)=0 \}$, 
and it can also be regarded as a stability estimate
for the functional $\Phi_m$. Its proof involves some  properties of Mellin transform.
Hereafter, we refer to $\gamma_0$ as the smallest number such that
  $$C_\gamma>0,\;\;\forall\gamma>\gamma_0.$$
  
\smallskip

Next,  we present a continuity result for the operator $I_m.$  
\begin{theorem}\label{B_TCont01} Let $\rho:[0,L]\to\R$ be a function in $L^2(0,L)$. Then,  for $ \gamma\geq \frac 3 4$   there exists  a positive constant ${C}_1>0$ such that 
\beq\label{B_Tcont01E00}
\norm{I_m[\rho]}_{1,\gamma,L_m^2}\leq  C_1\norm{\rho}_{L^2(0,L)},
\eeq
where $ C_1$  depends only on $L,\alpha_1,\alpha_{m-1},\alpha_{m},a_m$ and $\gamma.$
\end{theorem} 

Besides,  we present a stability result for the operator $I_m.$ 

\begin{theorem}\label{B_Tstab01} Let $\rho:[0,L]\to\R$ be a function in $L^2(0,L)$. Then,   for 
any  $\gamma>\max\{\gamma_0,3/4\}$, there exists  a positive constant $C_2>0$ such that
\beq\label{B_Estab01}
\norm{\rho}_{-1,\gamma+1,L}\leq  C_2 \norm{I_m[\rho]}_{1,\frac{\gamma}{2}-\frac 1 4,L_m^2},
\eeq
where $ C_2$  depends only on  $L,C_\gamma>0$ and $\gamma.$
\end{theorem} 

Theorems \ref{B_TCont01} and \ref{B_Tstab01}  are consequences
of Theorems \ref{B_Tcont02} and \ref{B_Tstab02}, respectively.

\smallskip
Even if the proof of Theorem~\ref{B_Tiden02} is provided for any choice of the partition 
$\{\alpha_j\}_{j=1}^m$  of $[0,c_0]$, its proof can be  considerably simplified in the special case when 
\beq\label{B_Halpha02}
\alpha_j=c_0\textrm{erfc}\left(\frac{\beta_0\beta^j}{2\sqrt{D}}\right) \quad j=1,...,m,
\eeq
with $\beta \in(0,1)$ and $\beta_0>0$  constants. Note that the corresponding mesh is non-regular. 

In what follows, $I_m$ and  $\Phi_m$  are denoted by $ \tilde I_m$ and 
$\tilde\Phi_m$, respectively, when $\alpha_j$ is given by \eqref{B_Halpha02}.

\smallskip
For the reconstruction, we  introduce the function  
\beq \label{B_Dfun:g}
g(t)=\frac{\tilde{I}_m[\rho](t^2/\beta_0^2)-
\tilde I_m[\rho](L_m^2)}{J_0F(c_0)} \quad \forall t\in\left[0, \beta_0L_m\right).
\eeq
As mentioned in the Introduction, we look for a reconstruction algorithm and 
a numerical scheme to recover function $\rho$ from the measurement of  
$\tilde I_m[\rho].$ We begin by recovering $\tilde\varphi:[0,L]\to\R$, which satisfies 
\beq
\tilde\Phi_m[\tilde\varphi](t/\beta_0)=g(t),\;\;\;\forall t\in[0,\beta_0L_m).
\eeq
Next, we define functions $\varphi_1,\varphi_2,...,\varphi_m$ by means of the 
following induction formulae: 
\beq\label{B_Dfunc:g1}
\varphi_1(x)=\left\{
\begin{array}{cc}
\D \frac{1}{a_m}g\left(\frac{x}{\beta^m}\right),&
\textrm{if } x \in[\beta L,L),\\ \\
0, & \textrm{otherwise,} 
\end{array}\right.
\eeq
and 
\beq\label{B_Dfunc:gj01}
\varphi_{k+1}(x)=\left\{\begin{array}{cc}
\D\frac{1}{a_m}\left(g\left(\frac{x}{\beta^m}\right)-\sum_{j=1}^k a_{m-k-1+j} \varphi_{j}\left(\frac{\beta^{j} x}{\beta^{k+1}}\right)\right), &\textrm{if } x\in[\beta^{k+1}L,\beta^k L), \\ \\
0,& \textrm{otherwise,}
\end{array}\right.
\eeq
for $k=1,..,m-1.$ 
Furthermore for $k\geq m$, we define 
\beq\label{B_Dfunc:gj02}
\varphi_{k+1}(x)=\left\{\begin{array}{cc}
\D\frac{1}{a_m}\left(g\left(\frac{x}{\beta^m}\right)-\sum_{j=1}^{m-1} a_{j} \varphi_{j+k-m+1}\left(\frac{\beta^{j} x}{\beta^{m}}\right)\right), &\textrm{if } x\in[\beta^{k+1}L,\beta^k L), \\ \\
0,& \textrm{otherwise.}
\end{array}\right.
\eeq

With the above definitions we have the following reconstruction result: 
\begin{theorem}\label{B_Trec} 
Let $\rho$ be a function in $C^0([0,L]),$  let  $g$  be defined as in  \eqref{B_Dfun:g}, and  let $\{\varphi_j\}_{j\geq 1}$ be given by  \eqref{B_Dfunc:g1}-\eqref{B_Dfunc:gj02}. Then  the function $\widetilde{\varphi}$ defined by 
\beq\label{B_TrecE00}
\tilde\varphi(x)=\left\{
\begin{array}{lcl}
\D \sum_{j=1}^{+\infty}\varphi_j(x), & & \textrm{ if } x\in(0,L], \\ \\
g(0), & & \textrm{ if } x=0,
\end{array}\right.
\eeq
is well defined and satisfies 
 \beq\label{B_TrecE01}
 \tilde \Phi_m[\tilde\varphi](t/\beta_0)=g(t) \quad\forall t\in[0,\beta_0L_m].
 \eeq
 Furthermore, $\rho$ satisfies   
 \beq \label{B_TrecE02}
 \int_0^x\rho(z)dz= \tilde\varphi(x)+\frac{\tilde I_m[\rho](L_m^2)}{J_0F(c_0)} \quad\forall x\in[0,L].
 \eeq
\end{theorem}

Theorem \ref{B_Trec}  provides an {\em explicit} reconstruction procedure for both operators 
$\tilde\Phi_m$ and $\tilde I_m$ and therefore a numerical algorithm for the reconstruction.

\begin{remark}
Theorem \ref{B_Trec} allows the recovery of $\varphi$, solution of  
\eqref{B_TrecE01}, without any restriction about $g$.  
If another mesh is substituted to the mesh given in \eqref{B_Halpha02}, the recovery of 
$\varphi$ imposes to do some assumptions about $g$. 
\end{remark} 

\smallskip
The previous reconstruction procedure gives us  the possibility to obtain a 
sharper stability result. We shall provide a stability result for 
$\tilde\Phi_m$ in terms of a quite general norm. 

\smallskip
We consider a family of norms  
$\norm{\cdot}_{[a,b)}$ for (some) functions $f:[a,b)\to \R $, where $0\le a< b<\infty$, that enjoys the following properties:
\begin{itemize}
\item[(i)] $\norm{f}_{[a,b)} <\infty$ for any $f\in W^{1,1}(a,b)$; 
\item[(ii)] If $[a_1,b_1)\subset [a,b)$, then   
\beq\label{B_Dnorm01}
\norm{f}_{[a_1,b_1)}\leq  \norm{f}_{[a,b)};
\eeq
\item[(iii)] For any $\lambda>0$, there exists a positive constant $C(\lambda)$ such that 
\beq\label{B_Dnorm02}
\norm{g_\lambda}_{[\lambda a,\lambda b)}\leq C(\lambda)\norm{f}_{[a,b)},
\eeq
where $g_\lambda(x)=f(x/\lambda),$ and $C(\cdot)$ is a nondecreasing function with $C(1)=1$.  
\end{itemize}

\smallskip
A natural family of norms fulfilling (i), (ii), and (iii), is those  of $L^p$ norms, where $1\le p\le +\infty$. Indeed, (i) and (ii) are obvious, and 
(iii) holds with 
$$
C(\lambda)=\left\{\begin{array}{cl}
\lambda^{\frac{1}{p}}& \textrm{ if } p\in [1,+\infty),\\
1&\textrm{ if } p=\infty.
\end{array}\right.
$$ 
Another family of norms fulfilling (i), (ii), and (iii), is the family of BV-norms: 
\beq
\norm{f}_{BV(a,b)}=\norm{f}_{L^\infty(a,b)}+
\sup_{a\leq x_1< \cdot\cdot\cdot< x_k<b}\sum_{j=1}^k\left|f(x_k)-f(x_{k-1}) \right| . 
\eeq
Here, we can pick $C(\lambda)=1$. (Note that $W^{1,1}(a,b)\subset BV(a,b)$, see e.g. \cite{EG}.)
These kinds of norms are adapted to functions with low regularity, as e.g.  step functions. 
The second main result in this paper is the following stability result.
\begin{theorem}\label{B_Tstab03} 
Let $\rho\in C^0([0,L])$ be a function and let a family of norms 
satisfying conditions (i), (ii) and (iii).
Then, we have  for all $k\ge 0$  
\beq\label{B_Tstab03E00}
\norm{\varphi(\cdot)-\varphi(L)}_{[\beta^{k+1}L,\beta^kL)}\leq 
C(\beta_0)\frac{C(\beta^m)}{a_m^{k+1}}\norm{\tilde\Phi_m[\varphi](\cdot)-
\tilde\Phi_m[\varphi](L_m)}_{[\beta^{k+1}L_m,L_m)},
\eeq
where $\displaystyle \varphi(x)=\int_0^x\rho(\tau)d\tau.$
\end{theorem}

Theorem \ref{B_Tstab03} 
shows in particular that the value of $\varphi$ in the interval 
$[\beta^{k+1}L,\beta^k L)$ depends on the value of $\tilde \Phi_m[\varphi]$ 
in the interval $[\beta^{k+1}L_m,L_m)$, a property which is closely related to the nature 
of the reconstruction procedure.  

\section{Proof of identifiability results}
\
\par

This section is devoted to proving the identifiability results for the operator
$\Phi_m$. 

\begin{proof}[\bf Proof of Theorem \ref{B_Tiden01}]

\smallskip
 Let $\varphi$ be an analytic function such that
$$
\Phi_m[\varphi](t)= \sum_{j=1}^n a_j\varphi(h_j(t))=0 \quad\forall t\in(0,\delta).
$$

Then, taking  $t\in(0,\min\{\delta,L_1\})$ and using the fact that 
\beq\label{B_Erela:L}
L_0<L_1<\cdot\cdot\cdot<L_m,
\eeq
we see that $h_j(t)=\beta_j t,$ $j=1,...,m.$ Then, we have
$$ 
\sum_{j=1}^m a_j\varphi(\beta_j t)=0, \qquad t\in (0,\min \{ \delta , L_1\} ).
$$ 

If we derive the above expression and evaluate it at zero, we obtain 
$$
\varphi^{(k)}(0)\left(\sum_{j=1}^m a_j (\beta_j )^k\right)=0 \quad\forall k\geq 0,
$$
where $\varphi^{(k)}(0)$ denotes the  $k-$th derivative of $\varphi$ at zero. 
Since $a_j,\beta_j$ are positive, we have that $\sum_{j=1}^m a_j (\beta_j )^k> 0;$  
therefore $\varphi^{(k)}(0)=0$  for all $ k\geq 0$, and hence $\varphi\equiv 0$. 
This proves the identifiability for $\Phi_m$ in the case of analytic functions.
\end{proof}

\smallskip
To prove Theorem \ref{B_Tiden02}, we need  some technical  lemmas. 

\begin{lemma}\label{B_Lemma01}
Let $f,g:[0,L]\to \R$ be  functions, and let $s,\alpha_0\in[0,1)$ and $\lambda\in(0,1)$ 
be numbers such that 
\beq\label{B_E01lemma01}
f(\tau)+g(\lambda\tau)=0 \quad\forall \tau\in[sL,L),
\eeq
and 
\beq\label{B_E02lemma01}
f(\tau) = 0 \quad\forall \tau\in[\alpha_0L,L).
\eeq
Then 
\beq\label{B_E03lemma01}
g(\tau)=0 \quad\forall \tau\in[\alpha_1L,\lambda L),
\eeq
where $\alpha_1=\lambda\max\{s,\alpha_0\}.$
\end{lemma}

Lemma \ref{B_Lemma01} is a direct consequence of  \eqref{B_E01lemma01} and \eqref{B_E02lemma01}.

\begin{lemma}\label{B_Lemma03} 
Let $f:[0,L]\to \R$ be a function, and let $s,\alpha_0\in[0,1)$ and $\lambda\in(0,1)$ be some numbers  such that
\beq\label{B_L03E01}
 f(\tau)=0 \quad\forall \tau\in[\tilde\alpha_kL,L) \quad\forall k\geq 1,
 \eeq
 where 
 \beq\label{B_L03E02} 
 \tilde\alpha_k=\lambda\max\{s,\tilde\alpha_{k-1}\} \quad\forall k\geq 1,
 \eeq
 with $\tilde\alpha_0=\alpha_0$.

\smallskip
Then, if $s>0,$ 
 $$
 f(\tau)=0 \quad\forall\tau\in[s\lambda L,L),
 $$
 and if $s=0,$
 $$
 f(\tau)=0   \quad\forall\tau\in(0,L).
 $$
\end{lemma}
\begin{proof}
To prove the above lemma, we need to consider two cases: $s=0$ and $s>0$. 

\smallskip

If $s>0$,   we claim that there exists $k_0$ such that  $\tilde\alpha_{k_0}<s$.
 Otherwise, if $\tilde\alpha_k\geq s\  \forall k\ge 0$, replacing in \eqref{B_L03E02}, we have 
 $$
 \tilde\alpha_{k+1}=\lambda\tilde\alpha_k,
 $$
 and hence $\tilde\alpha_k=\tilde\alpha_0\lambda ^k\to 0$, 
which is impossible, for $s>0$.
 
Using \eqref{B_L03E02}, the desired result follows, since
$$
\tilde\alpha_{k}=\lambda s \quad\forall k>k_0.
$$
   
\smallskip
 Now, if $s=0,$ replacing it in \eqref{B_L03E02} we obtain 
$$
    \tilde\alpha_k=\alpha_0\lambda^k.
$$

Then, using \eqref{B_L03E01} we have  
$$
f(\tau)=0,\quad\forall \tau\in (0,L),
$$
which completes the proof. 
\end{proof} 

\begin{lemma}\label{B_Lemma02}
Let $f:[0,L]\to \R$ be a function, and let $s,\alpha_0\in[0,1) $,  
$\lambda_1,...,\lambda_n\in(0,1)$ and $a_k>0,$ $k=0,...,n$ be some numbers such that 
$\lambda_1>\lambda_2>\cdot\cdot\cdot>\lambda_n\geq\alpha_0,$  
and 
\beq\label{B_E01lemma02}
a_0f(t)+\sum_{j=1}^na_j f(\lambda_jt)=0  \quad\forall t\in[sL,L),
\eeq 
and 
\beq\label{B_E02lemma02}
f(\tau) =  0\quad\forall \tau\in[\alpha_0L,L).
\eeq
Then  
\beq
\label{B_E03lemma02}
f(\tau)=0    \quad\forall \tau\in[\overline{\alpha}L,L),
\eeq
where $\overline{\alpha}=\lambda_ns.$ 
\end{lemma}

\begin{proof} We prove this result by induction on $n$.
\\ \\
{\bf Case $n=1$.}
In this case, from \eqref{B_E01lemma02} we have the following equations  
\beq\label{B_L02E01}
a_0f(t)+a_1f(\lambda_1 t)=0 \quad\forall t\in[sL,L),
\eeq
 \beq
 f(\tau)=0 \quad\forall \tau\in[\alpha_0L,L),
 \eeq
and $\alpha_0\leq\lambda_1.$
Then, applying Lemma \ref{B_Lemma01} with $g=f$, we get
$$
f(\tau)=0   \quad\forall \tau\in[\alpha_1L,\lambda_1L),
$$
where $\alpha_1=\lambda_1\max\{ s, \alpha_0 \}$, and thus
$$
f(\tau)=0    \quad\forall \tau\in[\alpha_1L,L),
$$ 
for  $\alpha_0\leq\lambda_1$.
 
\smallskip 
If $\alpha_0=0,$ we obtain the desired result: 
  $$
  f(\tau)=0   \quad\forall \tau\in[\lambda_1 sL,L).
  $$
  
On the other hand, when $\alpha_0>0,$ we can apply
Lemma \ref{B_Lemma01}  again   
with $\alpha_0$ replaced by  $\alpha_1$, since we have 
$$
a_0f(t)+a_1f(\lambda_1 t)=0   \quad\forall t\in[sL,L),
$$
$$
f(\tau)=0 \quad\forall \tau\in[\alpha_1L,L),
$$
and $\alpha_1\leq\lambda_1.$ Thus, we get by induction on $k \ge 0$ 
 \beq\label{B_L02E02}
 f(\tau)=0 \quad\forall \tau\in[\alpha_kL,L), \quad\forall k\geq 1,
 \eeq
 where 
 \beq\label{B_L02E03} 
 \alpha_k=\lambda_1\max\{s,\alpha_{k-1}\} \quad\forall k\geq 1.
 \eeq 
Note that, if $s=0,$ letting $t=0$  in \eqref{B_L02E01} yields $f(0)=0$. 
 Using Lemma \ref{B_Lemma03} with \eqref{B_L02E02}-\eqref{B_L02E03}, 
 we conclude that
 $$
 f(\tau)=0  \quad\forall \tau\in [\lambda_1 sL,L),
 $$
 which completes the case $n=1.$
\\ \\

{\bf Case $n+1.$} Assume the lemma proved up to the value $n$, and let us prove it for the value $n+1$.

Assume given a function $f : [0,L]\to \R $ and some numbers $s,\alpha _0 \in [0,1)$, $a_k>0$ for $0\le k\le n+1$, 
$\lambda _1 , ... ,\lambda _{n+1}\in (0,1)$ with 
$1>\lambda_1>\lambda_2> \cdots >\lambda_{n+1}\ge \alpha _0$, and such that
 \beq\label{B_L02E04}
a_0f(t)+\sum_{j=1}^{n+1}a_j f(\lambda_jt)=0 \quad\forall t\in[sL,L),
\eeq
and 
\beq\label{B_L02E05}
f(\tau)\equiv 0 \quad\forall \tau\in[\alpha_0L,L).
\eeq
Then we aim to prove that 
$$
f(\tau) =0 \quad\forall \tau\in[\lambda_{n+1}sL,L).
$$

We introduce the function  
$$
\psi(\tau)=\sum_{j=1}^{n+1} a_j f(\frac{\lambda_j}{\lambda_1}\tau) =
a_1f(\tau)+\sum_{j=2}^{n+1}a_j f(\tilde{\lambda}_j \tau),
$$ 
where $\displaystyle \tilde{\lambda}_j=\frac{\lambda_j}{\lambda_1}, \  j=2,..,n+1$. 

Then, using 
\eqref{B_L02E05}, we have  
\beq\label{B_L02E06}
\psi(\tau)=0   \quad\forall \tau\in[\lambda_1\frac{\alpha_0}{\lambda_{n+1}}L,L).
\eeq   

On the other hand, from (\ref{B_L02E04}), we have 
$$
a_0f(\tau)+\psi(\lambda_1\tau)=0 \quad\forall \tau\in[sL,L).
$$

Then, from \eqref{B_L02E05} and Lemma \ref{B_Lemma01}
with $g=\psi$, we conclude
$$
\psi(\tau)=0 \quad\forall \tau\in[\lambda_1\max\{\alpha_0,s\}L,\lambda_1L).
$$
Next, we set $s_1=\lambda_1 \max\{\alpha_0,s\} \in [0,1)$. Using \eqref{B_L02E06},
we have $\psi \equiv 0$ on 
$[s_1L,\lambda_1L)\cup [\lambda_1\frac{\alpha_0}{\lambda_{n+1}}L,L)$.
Therefore, with $\frac{\alpha _0 }{\lambda _{n+1} } \le 1$,
 \beq \label{B_L02E07}
 \psi(\tau)=a_1f(\tau)+\sum_{i=2}^{n+1}a_i f(\tilde{\lambda}_i \tau)=0 
 \quad\forall \tau \in[s_1L,L).
 \eeq
 
Note that $1>\tilde{\lambda}_2>\tilde{\lambda}_3>\cdot\cdot\cdot>\tilde{\lambda}_{n+1},$ 
and that $\alpha_0\leq \lambda_{n+1}<\frac{\lambda_{n+1}}{\lambda_1} = 
\tilde\lambda_{n+1}.$ Then, by using the induction hypothesis with (\ref{B_L02E07}) 
and \eqref{B_L02E05}, we obtain
$$
f(\tau)=0 \quad\forall\tau\in [\alpha_1L,L),
$$
where $\tilde\alpha_1 = s_1\tilde\lambda_{n+1} = 
\lambda_{n+1}\max\{s,\alpha_0\}<\lambda_{n+1}.$ Then we can repeat the latter 
argument replacing $\alpha_0$ by $\tilde\alpha_1,$ and we obtain 
 $$
 f(\tau)=0 \quad\forall \tau\in[\tilde\alpha_kL,L) \quad\forall k\geq 1,
 $$ where 
 \beq\label{B_L02E08} 
 \tilde\alpha_{k}=\lambda_{n+1}\max\{s,\tilde\alpha_{k-1}\} \quad\forall k\geq 1,
 \eeq 
 with $\tilde\alpha_0=\alpha_0$ given. 
 If $s=0$, letting $t=0$ in \eqref{B_L02E04} yields $f(0)=0$. Using 
 Lemma \ref{B_Lemma03}  we infer that 
$$
f(\tau)=0,\quad\forall \tau\in[\overline{\alpha}L,L),
$$ 
where $\overline{\alpha}=\lambda_{n+1}s,$ which completes the proof.
\end{proof}

\begin{proof}[\bf Proof of Theorem \ref{B_Tiden02}] 
Let  $\varphi:[0,L]\to\R$ be a function such that
$$ 
\Phi_m[\varphi](t)=\sum_{j=1}^m a_j \varphi(h_j(t))=0
\quad\forall t\in[0,L_m].
$$  

Then, if  $t=L_m,$ we obtain  
$$
h_j(L_m)=L \quad \forall j=1,...,m,
$$ 
and hence 
\beq\label{B_Tiden02E01}
0=\Phi_m[\varphi](L_m)=\varphi(L).
\eeq

Next,  for any $k\in \{ 1, .... , m\} $, we have
$$
\sum_{j=k}^m a_j\varphi(\beta_jt)=0 \quad\forall t\in[L_{k-1},L_k],
$$
which is equivalent to
\beq\label{B_Tiden02E03}
a_k\varphi(t) + 
\sum_{j=k+1}^m a_j\varphi\left(\frac{\beta_j}{\beta_k}t\right)=0
\quad\forall t\in[\beta_kL_{k-1},\beta_kL_k]=[\beta_kL_{k-1},L],
\eeq
for $k=1,2,...,m.$ We aim to prove that 
$$
\varphi(\tau)=0  \quad\forall \tau\in[\beta_mL_{k-1},L],
$$
for $k=1,...,m$. We proceed by induction on $i=m-k\in\{0,...,m-1\}$.
 \\ \\
{\bf Case  $i=0$.} Letting  $k=m$  in \eqref{B_Tiden02E03} yields
$$
a_{m}\varphi(t)=0 \quad\forall t\in[\beta_mL_{m-1},L],
$$
which implies 
\beq\label{B_Tiden02E04}
\varphi(\tau)=0 \quad\forall \tau\in[\beta_mL_{m-1},L],         
\eeq
which  completes  the case $i=0$.
\\\\
{\bf Case  $i=1$.} Letting  $k=m-1$  in \eqref{B_Tiden02E03}, we obtain
\beq\label{B_Tiden02E05}
a_{m-1}\varphi(t)+a_m\varphi\left(\frac{\beta_m}{\beta_{m-1}} t\right)=0
\quad\forall t\in[\beta_{m-1}L_{m-2},L].
\eeq
We infer from Lemma \ref{B_Lemma02} (applied with 
$\lambda_1=\frac{\beta_m}{\beta_{m-1}},$ 
$s=\frac{\beta_{m-1}}{\beta_{m-2}}$  and 
$\alpha_0=\frac{\beta_m}{\beta_{m-1}}$)  that
$$
\varphi(\tau)=0 \quad\forall \tau\in[\beta_mL_{m-2},L].
$$
\\ \\
{\bf Case i.} Assume the property satisfied   for $i-1$, i.e.,
\beq\label{B_Tiden02E06}
\varphi(\tau)=0 \quad\forall \tau\in[\beta_mL_{m-i},L].
\eeq
Replacing  $k=m-i$ in \eqref{B_Tiden02E03}, we obtain
\beq\label{B_Tiden02E07}
a_{m-i}\varphi(t)+\sum_{j=m-i+1}^m a_j\varphi\left(\frac{\beta_j}{\beta_{m-i}}t\right)=0
\quad\forall t\in[\beta_{m-i}L_{m-i-1},L].
\eeq

Then, if we set $\lambda_j=\frac{\beta_j}{\beta_{m-i}}<1,$ for $j=m-i+1,...,m,$ 
$$
s=\beta_{m-i}\frac{L_{m-i-1}}{L}
$$  
and $\displaystyle \alpha_0=\frac{\beta_m}{\beta_{m-i}}=\lambda_m,$ 
then we infer from Lemma \ref{B_Lemma02} that
$$
\varphi(\tau)= 0 \quad\forall \tau\in[\beta_{m}L_{m-i-1},L].
$$

Thus
$$
\varphi(\tau)=0 \quad\forall \tau\in[\beta_{m}L_{k-1},L],
$$
and for $k=1 , ...  , m$. This implies (with $k=1$ and $L_0=0$)
$$
\varphi(\tau)=0 \quad\forall \tau\in[0,L].
$$
The proof of Theorem \ref{B_Tiden02} is complete.
\end{proof}


\section{Proofs of the stability results}
\
\par

We first prove Theorem \ref{B_Tcont02}.
\begin{proof}[\bf Proof of Theorem \ref{B_Tcont02}] 
First, some estimates are established.
\smallskip
\bea\label{B_PTcont02E01}
\D\norm{\varphi\circ h_j}^2_{L^2(0,L_m)}& = &
\D\int_0^{L_m}\varphi^2(h_j(t))dt\nonumber
=\D\int_0^{L_j}\varphi^2(\beta_j t)dt+\varphi^2(L)L\left(\frac{1}{\beta_m }-
\frac{1}{\beta_j}\right) \nonumber\\
&\leq&\D\frac{1}{\beta_j}\int_0^{L}\varphi^2(t)dt+\varphi^2(L)\frac{L}{\beta_m} \nonumber
\leq\D\frac{1}{\beta_m}\left\{ \norm{\varphi}^2_{L^2(0,L)}+\varphi^2(L)L\right\}\nonumber \\
&\leq&\D\frac{1}{\beta_m}\left( 1+  ||T_L||^2 L\right)\norm{\varphi}^2_{H^1(0,L)},
\eea
where $T_L(u)=u(L)$ is the trace operator in $H^1(0,L)$. 

Now, if we set 
 $$
 c_1 = \frac{1}{\sqrt{\beta_m}}\left( 1+ 
 ||T_L||^2  L\right)^{\frac{1}{2}},
 $$
 then using \eqref{B_PTcont02E01}, we obtain
 \beq\label{B_Tcont02E02}
 \norm{\Phi_m[\varphi]}_{L^2(0,L_m)}\leq \sum_{j=1}^ma_j\norm{\varphi\circ h_j}_{L^2(0,L_m)}
 \leq c_1\norm{\varphi}_{H^1(0,L)}.
 \eeq
 
 On the other hand, let  $\psi$ be any test function with compact support in $(0,L_m).$ Then
\bea
 \D\int_0^{L_m}\Phi_m[\varphi](t)\psi'(t)dt
 &=&\D\sum_{j=1}^ma_j\left\{ \int_0^{L_j}\varphi(\beta_j t)\psi'(t)dt + 
 \varphi(L)\int_{L_j}^{L_m}\psi'(t)dt \right\}\nonumber\\
 &=& -\D\sum_{j=1}^ma_j\beta_j\int_0^{L_j}\varphi'(\beta_j t)\psi(t)dt \\
 &=& -\D\sum_{j=1}^ma_j\beta_j\int_0^{L_m}\varphi'(\beta_j t)\psi(t)(1-
 H(\beta_jt-L))dt, \nonumber
 \eea
 where $H$ denotes Heaviside's function. Thus
 \beq\label{B_Tcont02E03}
( \Phi_m[\varphi])'(t)= 
\sum_{j=1}^ma_j\beta_j\varphi'(\beta_j t)(1-H(\beta_jt-L)) \quad\forall t\in (0,L_m).
\eeq

Therefore, for any $\varphi\in H^1(0,L),$ the function $\Phi_m[\varphi]$ belongs 
to $H^1(0,L_m)$. This, along with \eqref{B_Tcont02E03} yields
\beq\label{B_Tcont02E04}
\norm{\left(\Phi_m[\varphi]\right)'}_{L^2(0,L_m)}\leq 
\sum_{j=1}^ma_j\sqrt{\beta_j}\left(\int_0^{L}(\varphi')^2( t)dt\right)^{1/2}
\leq\sqrt{\beta_1}\norm{\varphi'}_{L^2(0,L)}.
\eeq

Combining \eqref{B_Tcont02E04} with equation \eqref{B_Tcont02E02}, we obtain 
$$
\norm{\Phi_m[\varphi]}_{1,0,L_m}\leq 
\tilde C_1\norm{\varphi}_{1,0,L},
$$
where $\tilde C_1=\sqrt{(c_1)^2+\beta_1}.$ 
The proof of Theorem \ref{B_Tcont02} is therefore complete.
\end{proof}

\smallskip
Now we proceed to the proof of  Theorem  \ref{B_Tstab02}. Before establishing  this stability result, we 
need recall well-known facts about  Mellin Transform (the reader is referred 
to  \cite{Bio:Titchmarsh} Chapter~VIII, for details).

\smallskip For any real numbers $\alpha <\beta$, 
let $ <\alpha,\beta>$ denote the open strip of complex 
numbers $s=\sigma+it$ ($\sigma , t\in\R$) such that $\alpha<\sigma<\beta.$ 
\begin{definition}[Mellin transform] Let $f$ be locally Lebesgue integrable over 
$(0,+\infty)$. The {Mellin transform} of $f$ is defined by 
$$
\mathcal{M}[f](s)=\int\limits_{0}^{+\infty}f(x)x^{s-1}dx \quad\forall s\in<\alpha,\beta>,
$$
where $<\alpha,\beta>$ is the largest open strip in which the integral converges 
(it is  called the fundamental strip). 
\end{definition}

\smallskip
\begin{lemma}\label{B_Lemma:MellinProperties} Let $f$ be locally 
Lebesgue integrable over $(0,+\infty).$ Then the following properties hold 
true:
\begin{enumerate}
\item Let $s_0\in \R$. Then for all $s$ such that 
$s+s_0\in<\alpha,\beta>$, we have  
$$
\mathcal{M}[f(x)](s+s_0) = \mathcal{M}[x^{s_0}f(x)](s).
$$
\item For any  $\beta\in\R$, if $g(x)=f(\beta x)$, then
$$
\mathcal{M}[g](s)=\beta^{-s}\mathcal{M}[f](s)\qquad \forall s \in <\alpha , \beta >.
$$
\end{enumerate}
\end{lemma}

\smallskip

\begin{definition}[Mellin transform as operator in $L^2$]
For functions in $L^2(0,+\infty)$ we define a linear operator $\tilde{\mathcal{M}}$ 
as 
$$
\begin{array}{rll}
\tilde{\mathcal{M}}:L^2(0,+\infty)&\longrightarrow& L^2(-\infty,+\infty),\\ \\
f&\longrightarrow& \tilde{\mathcal{M}}[f](s):=
\frac{1}{\sqrt{2\pi}}\mathcal{M}[f](\frac{1}{2}-is).
\end{array}
$$
\end{definition}

\smallskip
\begin{theorem}[Mellin inversion theorem] The operator $\tilde{\mathcal{M}}$ 
is invertible with inverse 
$$
\begin{array}{rll}
\tilde{\mathcal{M}}^{-1}:L^2(-\infty,+\infty)&\longrightarrow& L^2(0,+\infty),\\ \\
\varphi&\longrightarrow& \tilde{\mathcal{M}}^{-1}[\varphi](x):=
\frac{1}{\sqrt{2\pi}}\int_{-\infty}^{+\infty}x^{-\frac{1}{2}-is}\varphi(s)ds.
\end{array}
$$
Furthermore, this operator is an isometry; that is, 
$$
\norm{\tilde{\mathcal{M}}[f]}_{L^2(-\infty,\infty)} =
\norm{f}_{L^2(0,\infty)} \quad\forall f\in L^2(0,+\infty).
$$
\end{theorem}

\smallskip
\begin{proof}[\bf Proof of the Theorem \ref{B_Tstab02}] 
We note that for any  function  $f:[0,+\infty  [ \to\R$ such that supp$(f)\subset[0,L)$,
we have  
$$
f(h_j(t)) = f(\beta_j t).
$$

Thus, we obtain 
\beq\label{B_L06E01}
\Phi_m[f](t)=\sum_{j=1}^ma_jf(\beta_jt) \quad\forall t\geq 0,
\eeq
where $\{\beta_j\}_{j=1}^m$ has been defined in \eqref{B_Dbeta01}.

\smallskip

Pick any $\varphi \in C([0,L])$ and let  $g:[0,L_m]\to\R$ be  such that
\beq\label{B_L04E01_1}
\Phi_m[\varphi](t)=g(t) \quad\forall t\in[0,L_m].
\eeq
Define the functions
\beq\label{B_L04E03}
\tilde g(t)=\left\{
\begin{array}{ll}
g(t)-g(L_m) & 0\leq t\leq L_m, \\
\\
0 & t\geq L_m,
\end{array}\right., \ 
\tilde \varphi(t)=\left\{
\begin{array}{ll}
\varphi(t)-\varphi(L) & 0\leq t\leq L, \\
\\
0 & t\geq L.
\end{array}\right. 
\eeq

If we replace $t$ by $L_m$ in \eqref{B_L04E01_1}, 
we have the following compatibility condition
$$
\varphi(L)=g(L_m).
$$

Since $\Phi _m [1]=1$, we infer that
\beq\label{B_L04E04}
\Phi_m[\tilde\varphi](t)=\tilde g(t) \quad\forall t\geq 0.
\eeq
Letting $f=\tilde\varphi$ in \eqref{B_L06E01} yields
$$
\Phi_m[\tilde\varphi](t) = 
\sum_{j=1}^m a_j\tilde \varphi(\beta_j t) \quad\forall t\geq 0.
$$
It follows  from Lemma \ref{B_Lemma:MellinProperties} that
\begin{equation}
\label{B_L04E02} 
\mathcal{M}\left[\Phi_m[\tilde\varphi]\right](s) = 
\left(\sum_{j=1}^m a_j\beta_j^{-s}\right)\mathcal{M}[\tilde \varphi](s)
\quad\forall s\in <\alpha,\beta>,
\end{equation}
where $<\alpha , \beta >$  is the fundamental strip  associated with $\tilde\varphi.$

\smallskip
Let $\gamma>0$ be a fixed constant. Using \eqref{B_L04E02}  and 
Lemma~\ref{B_Lemma:MellinProperties}, we obtain 
\beq\label{B_L05E02}
\Lambda_m^\gamma(s)\left|\tilde{\mathcal{M}}[x^\gamma\tilde \varphi(x)](s)\right| = 
\left|\tilde{\mathcal{M}}\left[x^\gamma\Phi_m[\tilde\varphi](x)\right](s)\right|
\quad\forall s\in\R,
\eeq
where $\Lambda_m^\gamma$ has been defined in \eqref{B_DCons01}.
On the other hand,
\bea
\D\Lambda_m^\gamma(s) &\geq& \D a_m\beta_m^{-\gamma-\frac{1}{2}} -
\left|\sum_{j=1}^{m-1} a_j\beta_j^{-(\gamma+\frac{1}{2}-is)}\right|\geq 
a_m\beta_m^{-\gamma-\frac{1}{2}} - 
\sum_{j=1}^{m-1} a_j\beta_j^{-(\gamma+\frac{1}{2})} \nonumber\\ 
&\geq &\D a_m\beta_m^{-\gamma-\frac{1}{2}} - 
\beta_{m-1}^{-(\gamma+\frac{1}{2})} = 
\beta_m^{-\gamma-\frac{1}{2}}\left(a_m - 
\left(\frac{\beta_{m-1}}{\beta_{m}}\right)^{-(\gamma+\frac{1}{2})}\right).
\eea

Therefore, if we choose 
$$
\gamma> \frac{\ln(a_m)}{\ln(\frac{\beta_{m}}{\beta_{m-1}})}-\frac{1}{2},
$$
then
$$
\Lambda_m^\gamma(s)\geq\beta_m^{-\gamma-\frac{1}{2}}\left(a_m - 
\left(\frac{\beta_{m-1}}{\beta_{m}}\right)^{-(\gamma+\frac{1}{2})}\right)>0
\quad\forall s\in\R.$$

Thus, there exists $\gamma_0$ such that
$$
C_\gamma=\inf_{s\in\R}\Lambda_m^\gamma(s)>0
\quad\forall\gamma>\gamma_0.
$$

Therefore, using the fact that $\tilde{\mathcal  M}$ is an isometry and \eqref{B_L05E02}, we obtain
\beq\label{B_L05E03}
C_\gamma\norm{\tilde{\varphi}}_{0,\gamma,L}\leq 
\norm{\Phi_m[\tilde\varphi]}_{0,\gamma,L_m}
\eeq
which completes the proof of Theorem \ref{B_Tstab02}.
\end{proof}

\smallskip
We are now in a position to prove Theorems \ref{B_TCont01} and \ref{B_Tstab01}.

 \begin{proof} [\bf Proof of Theorem~\ref{B_TCont01}] 
Let us fix any $\gamma>0$ and let $\rho:[0,L]\to \R$ be a function in $L^2(0,L).$ 
From \eqref{rela:Phi:Im} we have 
\bea
(x^\gamma I_m[\rho](x))'& = &
\D\gamma x^{\gamma-1}I_m[\rho](x) + x^\gamma (I_m[\rho](x))' \nonumber\\ &=&
\displaystyle \D\gamma x^{\gamma-1}I_m[\rho](x) + 
\frac{x^{\gamma-\frac 12}J_0F(c_0)}{2} (\Phi_m[\varphi])'(\sqrt{x}),
\eea
where $\varphi(x)=\int\limits_0^x\rho(\tau)d\tau.$ (Note that $\varphi \in H^1(0,L)$.)  Since
$$
\int_0^{L_m^2}x^{2\gamma-1}  \left((\Phi_m[\varphi])'(\sqrt{x})\right)^2dx = 
2\int_0^{L_m}\tau^{4\gamma-1} \left((\Phi_m[\varphi])'(\tau)\right)^2d\tau = 
2\norm{(\Phi_m[\varphi])'}_{0,2\gamma-\frac12,L_m}^2,
$$
we have 
\bea\label{B_Tcont01E01}
\D\norm{I_m[\rho]}^2_{1,\gamma,L_m^2}
&\leq& 
\D\norm{I_m[\rho]}_{0,\gamma,L_m^2}^2 + 
\left(\gamma\norm{I_m[\rho]}_{0,\gamma-1,L_m^2} + 
\frac{|J_0F(c_0)|}{\sqrt{2}}\norm{(\Phi_m[\varphi])'}_{0,2\gamma-\frac12,L_m}\right)^2 
\nonumber\\  
\nonumber \\ 
&\leq &
\D\norm{I_m[\rho]}_{0,\gamma,L_m^2}  ^2 + 
2\gamma^2\norm{I_m[\rho]}_{0,\gamma-1,L_m^2}^2 + 
\left(J_0F(c_0)\right)^2\norm{(\Phi_m[\varphi])'}_{0,2\gamma-\frac12,L_m}^2  
\nonumber \\ 
\nonumber \\
&\leq &\D(L^2+2\gamma^2)\norm{I_m[\rho]}_{0,\gamma-1,L_m^2}^2 +\left(J_0F(c_0)\right)^2\norm{(\Phi_m[\varphi])'}_{0,2\gamma-\frac12,L_m}^2.
\eea

On other hand, using  \eqref{rela:Phi:Im} and the change of variable 
$\tau=x^2$, we have 
\bea
\label{B_Tcont01E02}
\D\norm{\Phi_m[\varphi]}^2_{0,2\gamma-\frac 3 2,L_m} &=&
\D\frac{1}{\left(F(c_0)J_0\right)^2}\int_0^{L_m} 
x^{4\gamma-3}\left(I_m[\rho](x^2)\right)^2dx\nonumber\\ 
&=& \D
\frac{1}{2\left(F(c_0)J_0\right)^2}\norm{I_m[\rho]}^2_{0,\gamma-1 ,L_m^2}.
\eea

By replacing  \eqref{B_Tcont01E02}  in \eqref{B_Tcont01E01}, we obtain 
$$
\norm{I_m[\rho]}_{1,\gamma,L_m^2}^2 \leq (L^2 + 2\gamma^2)2\left(F(c_0)J_0\right)^2 \norm{\Phi_m[\varphi]}_{0,2\gamma-\frac{3}{2},L_m}^2 + \left(F(c_0)J_0\right)^2\norm{(\Phi_m[\varphi])'}_{0,2\gamma-\frac12,L_m}^2,
$$
and assuming that $\gamma\geq\frac 3 4$, from Theorem \ref{B_Tcont02}, we have  
\begin{equation}
\begin{array}{rll}
\norm{I_m[\rho]}_{1,\gamma,L_m^2}& \leq & 
\sqrt{3L^2+4\gamma^2}J_0F(c_0)L^{2\gamma-\frac{3}{2}}
\norm{\Phi_m[\varphi]}_{H^1(0,L_m)}\\ \\
& \leq & \sqrt{3L^2+4\gamma^2}J_0F(c_0)L^{2\gamma- 
\frac{3}{2}}\tilde C_1\norm{\varphi}_{H^1(0,L)}.
\end{array}
\end{equation}

But, from Cauchy-Schwarz inequality we have 
$|\varphi(x)|\leq \sqrt{L}\norm{\rho}_{L^2(0,L)},$ and hence
$$
\norm{\varphi}_{H^1(0,L)}^2=\norm{\varphi}_{L^2(0,L)}^2 + 
\norm{\rho}_{L^2(0,L)}^2 \leq (L^2+1)\norm{\rho}_{L^2(0,L)}^2.
$$ 

Therefore, for any $\gamma\geq\frac 3 4$, we have 
$$
\norm{I_m[\rho]}_{1,\gamma,L_m^2}\leq  C_1\norm{\rho}_{L^2(0,L)},
$$
where 
$$
C_1 = \sqrt{3L^2 + 4\gamma^2}J_0F(c_0)L^{2\gamma-\frac{3}{2}}
\tilde C_1(L^2+1)^{1/2},
$$
and the proof of Theorem \ref{B_TCont01} is therefore finished.
\end{proof}

\smallskip
\begin{proof}[\bf Proof of Theorem \ref{B_Tstab01}] 
Let $\psi$ be  any test function compactly supported in $(0,L),$ 
and let $\gamma$ be a positive constant. Set 
$$
g_\gamma(x) = x^\gamma\rho(x), \qquad  \qquad 
\varphi(x) = \int_0^x\rho(\tau)d\tau
$$ 
and 
$$
\tilde\varphi(t)=\varphi(x)-\varphi(L).
$$
It follows that  
$$
(x^{\gamma+1}\tilde\varphi(x))' = 
(\gamma+1)x^\gamma\tilde\varphi(x) + g_{\gamma+1}(x),
$$
and hence,
\bea
\D<g_{\gamma+1},\psi>&=&
\D\int_0^Lg_{\gamma+1}(x)\psi(x)dx = 
\int_0^L\left((x^{\gamma+1}\tilde\varphi(x))' - 
(\gamma+1)x^{\gamma}\tilde\varphi(x)\right)\psi(x)dx\nonumber\\  &=&
\D -\int_0^L\left(x^{\gamma+1}\tilde\varphi(x)\psi'(x) + 
(\gamma+1)x^\gamma\tilde\varphi(x)\psi(x)\right)dx.\nonumber
\eea
Then, we have 
\bea
|<g_{\gamma+1},\psi>|&\leq &
\left(\norm{\tilde\varphi}_{0,\gamma+1,L} + 
(\gamma+1)\norm{\tilde\varphi}_{0,\gamma,L}\right)
\norm{\psi}_{H^1(0,L)}\nonumber\\ \nonumber \\  &\leq&
\left(L+\gamma+1\right)\norm{\tilde\varphi}_{0,\gamma,L}
\norm{\psi}_{H^1(0,L)}.\nonumber
\eea
Therefore, 
\beq\label{B_Tstab01E02}
\norm{g_{\gamma+1}}_{H^{-1}(0,L)}\leq 
\left(L+\gamma+1\right)\norm{\tilde\varphi}_{0,\gamma,L}.
\eeq
Thus,  using Theorem \ref{B_Tstab02}, we have that for any 
$\gamma>\max\{\gamma_0,\frac 3 4\}$
there exists a constant $C_\gamma>0$ such that
\begin{eqnarray}
\label{B_Tstab01E03}
&&\norm{\rho}_{-1,\gamma+1,L} =
\norm{g_{\gamma+1}}_{H^{-1}(0,L)}\!\!  \nonumber\\
&&\qquad \quad \leq\!\! 
\left(L+\gamma+1\right)C_\gamma ^{-1} \left\{\norm{\Phi_m[\varphi]}_{0,\gamma,L_m} + 
\frac{L_m^{\gamma + \frac{1}{2}} }{\sqrt{2\gamma +1}} |\Phi_m[\varphi](L_m)|\right\}.
\end{eqnarray}
Using  \eqref{rela:Phi:Im}, we have  
 \bea
 \label{eq:phi}
 \Phi_m[\varphi](L_m) = \frac{1}{F(c_0)J_0}I_m[\rho](L_m^2).
\eea
Replacing \eqref{eq:phi} in \eqref{B_Tstab01E03} and 
using \eqref{B_Tcont01E02}, with  $2\gamma-3/2$ replaced by $\gamma$, 
we obtain 
\bea
\D\norm{\rho}_{-1,\gamma+1,L}\leq\D
\frac{\left(L+\gamma+1\right)}{\sqrt{2}|J_0F(c_0)|}
C_\gamma ^{-1} \left\{1+\sqrt{2}\frac{L_m}{\sqrt{2\gamma +1 }} ||T_{L_m^2} ||   \right\}
\norm{I_m[\rho]}_{1,\frac{\gamma}{2}-\frac 1 4,L_m^2} .\nonumber
\eea
Therefore, setting  
$$
C_2 = \frac{\left(L+\gamma+1\right)}{\sqrt{2}|J_0F(c_0)|}
C_\gamma ^{-1} \left\{1+ \sqrt{2}\frac{L_m}{\sqrt{2\gamma +1 }} ||T_{L_m^2} ||  \right\},
$$
we obtain \eqref{B_Estab01}. The proof of Theorem \ref{B_Tstab01} is achieved.
\end{proof}

\section{Numerical reconstruction results}

This section is devoted to the proof of Theorems \ref{B_Trec} and \ref{B_Tstab03}.

\smallskip
\begin{proof} [\bf Proof of Theorem \ref{B_Trec}] 
Let $\rho$ be a function in $C^0([0,L]),$ and 
let us consider the functions $\{\varphi_j\}_{j\geq 1}$ defined in 
\eqref{B_Dfunc:g1}-\eqref{B_Dfunc:gj02}. 

\smallskip
First, we note that for all $k\ge 1$ we have
\beq\label{B_TrecE03}
\varphi_{k}(x)=0, \qquad \forall x\not\in  [\beta^{k}L,\beta^{k-1}L).
\eeq
Then, we can define the sequence $\{\psi_p\}_{p\in\N ^*}$ as 
$$
\psi_p(x)=\sum_{j=1}^{p}\varphi_j(x) \quad\forall x\in\R .
$$
Using \eqref{B_TrecE03} we have that for all $x\in \R\setminus (0,L)$, 
$$
\psi_p(x)=0 \quad\forall p\in\N ^*,
$$ 
and hence,
$$
\lim_{p\to+\infty}\psi_p(x) = 0 \quad \forall x\in\R\setminus(0,L).
$$
Besides that, we consider the {\em ceiling function} 
$$
\lceil x \rceil = \min\{k\in\mathbb{Z}\;\big|\; k\geq x\},
$$
i.e.,  $\lceil x \rceil$ is the smallest integer not less than $x$.

\smallskip
Next, we define   
\beq\label{B_TrecE04}
k^\ast(x)=\left\lceil\frac{\ln(x/L)}{\ln(\beta)}\right\rceil \quad\forall x\in(0,L).
\eeq
Then, we have 
$$
x\in[\beta^{k^\ast(x)}L,\beta^{k^\ast(x)-1}L),\;\;\forall x\in(0,L).
$$
Therefore, we obtain for $x\in(0,L)$ 
$$
\psi_p(x)=\varphi_{k^\ast(x)}(x) \quad \forall p\geq k^\ast(x),
$$
and hence, 
\beq\label{B_TrecE06}
\lim_{p\to+\infty}\psi_p(x)=\varphi_{k^\ast(x)}(x) \quad\forall x\in(0,L).\eeq
Thus, the series in  \eqref{B_TrecE00} is convergent, i.e.  the function $\tilde\varphi $ is well defined. 

\smallskip
On the other hand, by replacing \eqref{B_Halpha02} in \eqref{B_Dbeta01} we obtain 
$$
\beta_j = \beta_0\beta^j, \ j=1,...,m.
$$ 
By using \eqref{B_TrecE03} we have  
$$
\tilde\varphi (x)=0, \qquad \forall x\in \R \setminus [0,L),
$$
and combining with  \eqref{B_L06E01}, we get
\beq\label{B_TrecE05}
\tilde\Phi_m[\tilde\varphi](t/\beta_0) = 
\sum_{j=1}^ma_j\tilde\varphi(\beta^jt) \quad\forall t\geq 0.
\eeq
By replacing $t=0$ and $t=\beta_0L_m$ in  \eqref{B_TrecE05}, 
and using  \eqref{B_Dfun:g}, \eqref{B_Dfunc:g1}, and \eqref{B_TrecE00}, we obtain
$$
\tilde\Phi_m[\tilde\varphi](0)=g(0),\quad\textrm{and}\quad
\tilde\Phi_m[\tilde\varphi](L_m)=g(\beta_0L_m)=0.
$$
Now, if we take $t \in (0,\beta_0L_m)=(0,L/\beta^m)$, we have  
$$
\beta^jt\in(0,\beta^{j-m}L),\;\;\textrm{for } j \in \{1,2,...,m\}.
$$

We need to consider two cases,  $ t< L$ and $ t\geq L$.

\smallskip
{\bf Case $ t< L$.} In this case we have 
$$
\beta^jt\in(0, L),\;\;\textrm{for } j\in \{1,2,...,m\},
$$
and 
$$
k^\ast(\beta^jt) = j+k^\ast(t). 
$$
Thus, replacing \eqref{B_TrecE06} in \eqref{B_TrecE05}, we obtain
$$
\tilde\Phi_m[\tilde\varphi](t/\beta_0)=\sum_{j=1}^ma_j\varphi_{j+k^\ast(t)}(\beta^jt),
$$
and hence, using \eqref{B_Dfunc:gj02} with $k+1=m+k^\ast(t)$, we obtain
\bea
\tilde\Phi_m[\tilde\varphi](t/\beta_0)&=&
\D\sum_{j=1}^{m-1}a_j\varphi_{j+k^\ast(t)}(\beta^jt) + 
a_m\varphi_{m+k^\ast(t)}(\beta^mt)\nonumber\\ &=&
\D\sum_{j=1}^{m-1}a_j\varphi_{j+k^\ast(t)}(\beta^jt) + 
\left(g\left(t\right)-\sum_{j=1}^{m-1} a_{j} \varphi_{j+k^\ast(t)}
\left(\beta^{j} t\right)\right)\nonumber\\  &=&g(t).\nonumber
\eea
\\ \\
{\bf Case $t\geq L.$} Let us set 
$$
k_\ast(x) = \left\lfloor \frac{\ln(x/L)}{\ln(1/\beta)}\right\rfloor \quad\forall x\geq L,
$$
where 
$$
\lfloor x\rfloor = 
\max\{k\in\mathbb{Z}\;\Big| \; k\leq x\},
$$ 
is  the floor function, i.e., it is the largest integer not greater than $x$. 
Thus, we have 
\beq
\label{VW}
\beta^{k_\ast(x)+1}x<L\leq \beta^{k_\ast(x)}x\quad\forall x\geq L,\eeq
and 
\beq\label{B_TrecE07}
k^\ast(\beta^{k_\ast(t)+1}t)=1.
\eeq
Then, we infer from \eqref{VW} that
$$
\begin{array}{cc}
\beta^j t\geq L& \forall  j\leq k_\ast(t),\\ \\
\beta^j t< L & \forall j\geq k_\ast(t)+1.
\end{array}
$$
By using \eqref{B_TrecE05} and \eqref{B_TrecE06}, it follows that
\bea
\D\tilde\Phi_m[\tilde\varphi](t/\beta_0)&=&
\D\sum_{j=k_\ast(t)+1}^ma_j\varphi_{k^\ast(\beta^jt)}\left(\beta^jt\right)\nonumber\\ 
&=&
\D\sum_{j=1}^{m-k_\ast(t)}a_{j+k_\ast(t)}\varphi_{k^\ast(\beta^{j+k_\ast(t)}t)}
\left(\beta^{j+k_\ast(t)}t\right).\nonumber
\eea
From \eqref{B_TrecE07}, we have
$$
k^\ast(\beta^{j+k_\ast(t)}t)=j    \quad\forall j\geq 1,
$$
and hence, from \eqref{B_Dfunc:g1}-\eqref{B_Dfunc:gj01}, with $k+1=m-k_\ast(t)$, 
we obtain
\bea
\D\tilde\Phi_m[\tilde\varphi](t/\beta_0)
&=&\D\sum_{j=1}^{m-k_\ast(t)}a_{j+k_\ast(t)}\varphi_{j}\left(\beta^{j+k_\ast(t)}t\right) \nonumber\\ 
&=&\D\sum_{j=1}^{m-k_\ast(t)-1}a_{j+k_\ast(t)}\varphi_{j}\left(\beta^{j+k_\ast(t)}t\right)+a_{m}\varphi_{m-k_\ast(t)}\left(\beta^{m}t\right)\nonumber\\ 
&=&\D\sum_{j=1}^{m-k_\ast(t)-1}a_{j+k_\ast(t)}\varphi_{j}\left(\beta^{j+k_\ast(t)}t\right)\nonumber \\
&+&\D\left(g(t)-\sum_{j=1}^{m-k_\ast(t)-1}a_{k_\ast(t)+j}\varphi_j\left(\beta^{j+k_\ast(t)}t\right)\right)\nonumber\\ 
&=&g(t). \nonumber
\eea 
It remains to  prove \eqref{B_TrecE02}. 
Replacing \eqref{B_Dfun:g} in \eqref{B_TrecE01} and using \eqref{rela:Phi:Im}, we get
\bea\label{B_TrecE08}
\D\tilde\Phi_m[\tilde\varphi](t/\beta_0)&=&
\D\frac{\tilde{I}_m[\rho](t^2/\beta_0^2) - 
\tilde I_m[\rho](L_m^2)}{J_0F(c_0)} \nonumber\\
&=& 
\D \tilde\Phi_m[\varphi](t/\beta_0) - 
\frac{\tilde I_m[\rho](L_m^2)}{J_0F(c_0)} \quad\forall t\in\left[0, \beta_0L_m\right],
\eea
where $\varphi(x) = \int_0^x\rho(\tau)d\tau$. Using 
$\tilde\Phi _m[1]=1$ and  Theorem~\ref{B_Tiden02} we obtain \eqref{B_TrecE02}, i.e.
$$
\tilde\varphi(x) = \varphi(x)-\frac{\tilde I_m[\rho](L_m^2)}{J_0F(c_0)} 
\quad \forall x\in\left[0, L\right].
$$  
This completes the proof of Theorem \ref{B_Trec}.
\end{proof}

\smallskip
\begin{proof} [\bf Proof of Theorem \ref{B_Tstab03}] 
\label{B_Tnorms}
Let $\rho$ be a function in ${C}^0([0,L]),$ and 
let $\{\varphi_j\}_{j\geq 1}$ be defined in \eqref{B_Dfunc:g1}-\eqref{B_Dfunc:gj02}.
Using \eqref{B_TrecE02}, we obtain
\beq\label{B_TstabE001}
\tilde \varphi(x)=\varphi(x)-\varphi(L) \quad\forall x \in[0,L],
\eeq
where $\varphi=\int_0^x\rho(\tau)d\tau$ and 
$\tilde\varphi$ has been defined in \eqref{B_TrecE00}.

Recall that the family of norms $ || \cdot ||_{ [a,b) }$ (for $0\le a<b<\infty$)   satisfies \eqref{B_Dnorm01}-\eqref{B_Dnorm02}.   
Using \eqref{B_TrecE00}, \eqref{B_TrecE04}-\eqref{B_TrecE06}
 and \eqref{B_TstabE001},  we obtain 
\beq\label{B_Tstab03E01}
\norm{\varphi(\cdot)-\varphi(L)}_{[\beta^{k+1}L,\beta^kL)} = 
\norm{\varphi_{k+1}}_{[\beta^{k+1}L,\beta^kL)}.
\eeq
Let us prove that for any $k\geq 0,$ we have
\beq\label{B_Tstab03E02}
 \norm{\varphi_{k+1}}_{[\beta^{k+1}L,\beta^kL)}\leq 
 \frac{C(\beta^m)}{a_m^{k+1}}\norm{g}_{[\beta^{k+1}\beta_0L_m,\beta_0L_m)}.
 \eeq
The proof of   \eqref{B_Tstab03E02} is done by  induction on $k$.
\\ \\
{\bf Case $k=0$.} Using \eqref{B_Dfunc:g1} and 
\eqref{B_Dnorm01}-\eqref{B_Dnorm02}, we have 
$$
\norm{\varphi_{1}}_{[\beta L,L)}\leq 
\frac{C(\beta^m)}{a_m}\norm{g}_{[\beta \beta_0L_m,\beta_0L_m)},
 $$
 as desired.

Assume now that for all $j=1,...,k$ (with $k\ge 1$), we have 
 \beq\label{B_Tstab03E03}
 \norm{\varphi_{j}}_{[\beta^j L,\beta^{j-1}L)}\leq 
 \frac{C(\beta^m)}{a_m^{j}}\norm{g}_{[\beta^{j}\beta_0L_m,\beta_0L_m)},
 \eeq
 and let us  prove \eqref{B_Tstab03E02}.
 \\ \\
 {\bf Case $k+1\leq m.$}  
 Using \eqref{B_Dfunc:gj01} and  \eqref{B_Dnorm01}-\eqref{B_Dnorm02}, 
 we obtain   
 \begin{multline*}
 \norm{\varphi_{k+1}}_{[\beta^{k+1}L,\beta^kL)} \leq 
 \frac{1}{a_m}\left(C(\beta^m)\norm{g}_{[\beta^{k+1}\beta_0L_m,\beta^k\beta_0L_m)} \right.\\ \left. + 
 \sum_{j=1}^k a_{m-k-1+j} C\left(\frac{\beta^{k+1}}{\beta^{j}}\right)
 \norm{\varphi_{j}}_{[\beta^jL,\beta^{j-1}L)}\right).
 \end{multline*}
 Using induction hypothesis \eqref{B_Tstab03E03},  we have 
 $$
 \begin{array}{l} 
\D \norm{\varphi_{k+1}}_{[\beta^{k+1}L,\beta^kL)}\leq\\ \\
\D \frac{1}{a_m}\left(C(\beta^m)\norm{g}_{[\beta^{k+1}\beta_0L_m,\beta^k\beta_0L_m)} + 
\sum_{j=1}^k a_{m-k-1+j} C\left(\frac{\beta^{k+1}}{\beta^{j}}\right)
\frac{C(\beta^m)}{a_m^{j}}\norm{g}_{[\beta^{j}\beta_0L_m,\beta_0L_m)}\right) \\ \\
 \leq
 \D\frac{C(\beta^m)}{a_m^{k+1}}\left(a_m^k\norm{g}_{[\beta^{k+1}\beta_0L_m,\beta^k\beta_0L_m)} +\sum_{j=1}^k a_{m-k-1+j} C\left(\frac{\beta^{k+1}}{\beta^{j}}\right)\norm{g}_{[\beta^{j}\beta_0L_m,\beta_0L_m)}\right)\\ \\
 \leq
 \D\frac{C(\beta^m)}{a_m^{k+1}}\left(a_m^k+\sum_{j=1}^k a_{m-k-1+j} C\left(\frac{\beta^{k+1}}{\beta^{j}}\right)\right)\norm{g}_{[\beta^{k+1}\beta_0L_m,\beta_0L_m)} 
 .\end{array}
 $$
Note  that
\beq\label{B_Tstab03E04}
C(u)\leq 1 \quad\forall u\in(0,1),\eeq
for $C(\cdot)$ is nondecreasing and $C(1)=1$. 
Therefore, 
$$
C\left(\frac{\beta^{k+1}}{\beta^{j}}\right)\leq 1\quad\forall j\in\{1,...,k\}.
$$
Thus, we obtain
$$
\D \norm{\varphi_{k+1}}_{([\beta ^{k+1} ,\beta ^k L)}\leq
\D\frac{C(\beta^m)}{a_m^{k+1}}\norm{g}_{ [\beta^{k+1}\beta_0L_m,\beta_0L_m)}.
$$
This proves \eqref{B_Tstab03E02} for all $k=\{0,...,m-1\}$.
\\ \\
{\bf Case $k+1>m$.}  Replacing  $\varphi _{k+1}$ by its expression in \eqref{B_Dfunc:gj02} and using \eqref{B_Dnorm01}-\eqref{B_Dnorm02} 
and the induction hypothesis, we obtain
$$
\begin{array}{l} 
\D \norm{\varphi_{k+1}}_{[\beta^{k+1}L,\beta^kL)}\leq\\ \\
\D \frac{1}{a_m}\left(C(\beta^m)\norm{g}_{[\beta^{k+1}\beta_0L_m,\beta^k\beta_0L_m)}+\sum_{j=1}^{m-1} a_{j} C\left(\frac{\beta^{m}}{\beta^{j}}\right)\frac{C(\beta^m)}{a_m^{j+k-m+1}}\norm{g}_{[\beta^{j+k-m+1}\beta_0L_m,\beta_0L_m)}\right) \\ \\
 =
 \D\frac{C(\beta^m)}{a_m^{k+1}}\left(a_m^k\norm{g}_{[\beta^{k+1}\beta_0L_m,\beta^k\beta_0L_m)}+\sum_{j=1}^{m-1} a_{j} C\left(\frac{\beta^{m}}{\beta^{j}}\right)\norm{g}_{[\beta^{j+k-m+1}\beta_0L_m,\beta_0L_m)}\right)\\ \\
 \leq
 \D\frac{C(\beta^m)}{a_m^{k+1}}\left(a_m^k+\sum_{j=1}^{m-1} a_{j} C\left(\frac{\beta^{m}}{\beta^{j}}\right)\right)\norm{g}_{[\beta^{k+1}\beta_0L_m,\beta_0L_m)} 
 ,\end{array}
$$
and with \eqref{B_Tstab03E04} we infer that 
$C({\beta^m}/{\beta^j})\leq 1$ for all $j\in\{1,...,m-1\}.$ 
This completes the proof of \eqref{B_Tstab03E02}.

\smallskip
On the other hand, using \eqref{B_TrecE01} and \eqref{B_Dnorm02}, we obtain
 $$
 \norm{g}_{[\beta^{k+1}\beta_0L_m,\beta_0L_m)} \le 
 C(\beta_0)\norm{\tilde\Phi_m[\tilde\varphi]}_{[\beta^{k+1}L_m,L_m)}.
 $$
By replacing \eqref{B_TstabE001} in \eqref{B_Tstab03E02}, we obtain
$$
\norm{\varphi_{k+1}}_{[\beta^{k+1}L,\beta^kL)}\leq C(\beta _0) 
\frac{C(\beta^m)}{a_m^{k+1}}
\norm{\tilde\Phi_m[\varphi](\cdot)-\tilde\Phi_m[\varphi](L_m)}_{[\beta^{k+1}L_m,L_m)},
$$
and by replacing in \eqref{B_Tstab03E01}, we obtain  \eqref{B_Tstab03E00}. 
This completes the proof of Theorem~\ref{B_Tstab03}.
 
\end{proof}

\section{Numerical results}
\
\par

In this section we discuss the numerical implementation of the scheme developed 
when proving Theorem~\ref{B_Trec}.

\smallskip
Firstly,  we define $\{\alpha\}_{j=1}^m$ as in \eqref{B_Halpha02}, 
and let
$$
F_j=\left\{
\begin{array}{ll}
\D 0&\;\; j=0, \\ \\
\D F\left(\frac{\alpha_j+\alpha_{j+1}}{2}\right)&\;\; j=1,...,m-1, \\ \\
\D F(c_0) &\;\; j=m,
\end{array}\right.
$$
where $F$ is Hill's function defined in \eqref{Int:Dhillfunc}.
Next, we set
\beq
a_j=\frac{F_j -F_{j-1}}{F(c_0)}, \qquad  j=1,...,m.
\eeq
Since $F$ is increasing and $0\leq F(x)< 1$ for all $x\geq 0,$ 
we infer that $a_j > 0$ for all $j=1,...,m$, and that
$$
\sum_{j=1}^ma_j=1.
$$ 

\smallskip
The corresponding approximation $F_m$ of  Hill's function is shown graphically in Figure \ref{fig:HillAprox}.
\begin{figure}[htc]
   \begin{center}
    \includegraphics[width=0.6\textwidth]{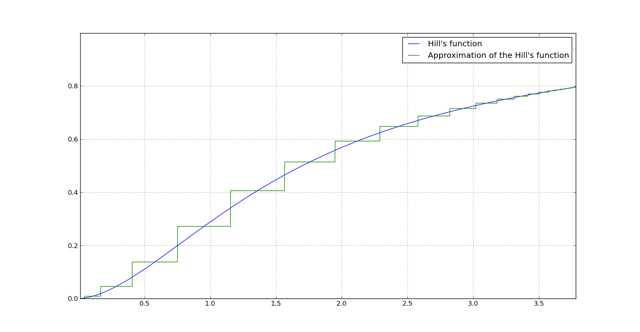}
  \caption{Hill's function and its approximation}
  \label{fig:HillAprox}
  \end{center}
\end{figure}

\smallskip
Recall that a non-regular mesh in the interval $[0,L]$ was introduced when proving
Theorem~\ref{B_Trec}. Now, let us start defining  
$$
\mathcal{P}_{q,1}=\left\{ (x_0,x_1, \ldots , x_q)\in \R^{q+1} 
\ : \;\; x_{j}\in[\beta L,L),\; x_0=\beta L,\;x_{j-1}<x_{j},\;\forall j=1,...,q\right\},
$$
and its representative vector 
$$
{\bf P}_{1}=\left(
x_0,
x_1,
\ldots ,
x_q
\right)\in\R^{q+1}.
$$ 
Next, introduce the sets
$$
\mathcal{P}_{q,j}=\left\{x\Big |\;\; \beta^{j-1} x\in \mathcal{P}_{q,1}\right\},\qquad j\ge 1
$$ 
and denote their corresponding representative vectors by
$$
{\bf P}_{j} = \beta^{j-1}{\bf P}_{1} = \beta{\bf P}_{j-1}\in\R^{q+1}.
$$ 
Let us fix some $p\ge 1$. Our aim is to recover the  function $\rho$  on the mesh 
\beq \label{nonRegularMesh}
\Sigma_{p,q}=\D\cup_{j=1}^p\mathcal{P}_{q,j},\eeq 
where the  corresponding representative vector is given by
$$
{\bf P}=\left(
{\bf P}_1,
{\bf P}_2,
\ldots ,
{\bf P}_{p}
\right)\in\R^{p+q+1}.
$$ 
By using \eqref{B_Dfunc:g1}-\eqref{B_Dfunc:gj02}, we can define the vectors  
${\bf G}_1,{\bf G}_2,..., {\bf G}_m\in\R^{q+1}$ inductively as follows:
\beq
\left({\bf G}_1\right)_s = 
\frac{1}{a_m}g\left(\frac{\left({\bf P}_{1}\right)_s}{\beta^m}\right),\;\; s=1,...,q+1,
\eeq
and for $k=1, ...  ,m-1$: 
\beq
\left({\bf G}_{k+1}\right)_s=
\D\frac{1}{a_m}\left(g\left(\frac{({\bf P}_{k+1})_s}{\beta^m}\right) - 
\sum_{j=1}^k a_{m-k-1+j} ({\bf G}_{j})_s\right) \quad  s=1,...,q+1.
\eeq 
Finally, we define the vectors ${\bf G}_k\in\R^{q +1}$ for $k= m,...,p-1,$ by  
\beq
({\bf G}_{k+1})_s = 
\frac{1}{a_m}\left(g\left(\frac{({\bf P}_{k+1})_s}{\beta^m}\right) - 
\sum_{j=1}^{m-1} a_{j} \left({\bf G}_{j+k-m+1}\right)_s\right) \quad s=1, ... , q+1. 
\eeq
Introduce the vector
$$
{\bf G}=\left(
{\bf G}_1,
{\bf G}_2,
\cdot
\cdot
\cdot,
{\bf G}_{p}
\right)\in\R^{p + q  + 1},
$$
which represents a discretization of the function $\tilde\varphi$ given by 
Theorem \ref{B_Trec} on the mesh defined by $\Sigma_{p,q},$; that is,
$\left(({\bf P})_s,({\bf G})_s\right)_{s=1}^{p+q+1}$ is a discretization of the 
curve $(x,\tilde\varphi(x)),$  $x\in(0,L).$ 

\smallskip
Therefore, using \eqref{B_TrecE02} and applying a forward difference scheme, 
we obtain an approximation of the curve $(x,\rho (x))$, $x\in (0,L)$,  through the vectors 
${\bf X},{\bf Y}\in\R^{p+q}$ given by
\begin{equation}
({\bf X})_s = 
({\bf P})_s,\quad ({\bf Y})_s = 
\max\left\{\frac{({\bf G})_{s+1}-({\bf G})_{s}}{({\bf P})_{s+1}-({\bf P})_{s}},0\right\}
\quad  s=1,...,p+q.
\label{ZYX}
\end{equation}
It should be noted that the maximum function was considered in \eqref{ZYX} because of  the positivity 
restriction on the density function. 

\smallskip
\subsection{Examples} 
Let us consider 
\beq\label{targetFunc}
\rho(x)=  \frac{8a^8x^7}{(x^8+a^8)^2},
\qquad
\varphi(x)=\int_0^x\rho(\tau)d\tau=  \frac{x^8}{x^8+a^8},
\eeq
with  $a=1.5$. Figure \ref{fig:reconstruc:Ejem} shows 
functions $\rho (x)$ and $\varphi (x)$ defined in   \eqref{targetFunc} and their approximations obtained by the previous 
procedure. 

\begin{figure}[h] 
\begin{center}
\begin{minipage}[c]{7cm} 
\centering 
 {\bf A}: Function $\rho$.
\end{minipage}
\begin{minipage}[c]{7cm} 
\centering 
{\bf B}:  Function $\varphi$.
\end{minipage}
\end{center}
\begin{minipage}[l]{7cm} 
\includegraphics[width=1.1\textwidth]{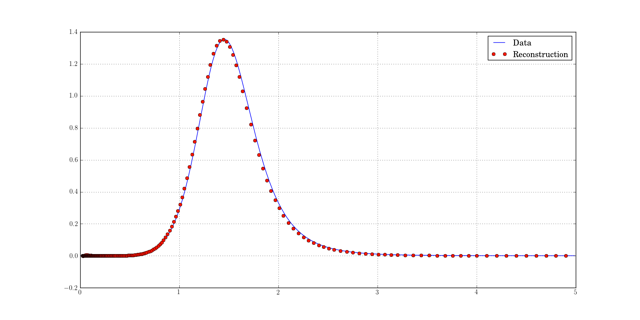} 
\end{minipage}
\begin{minipage}[r]{7cm} 
\includegraphics[width=1.1\textwidth]{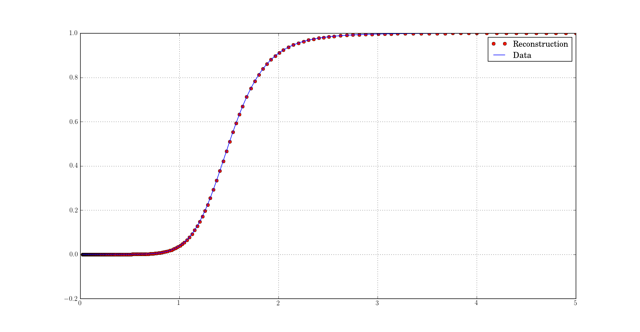} 
\end{minipage} 
\caption{The target functions $\rho$ and $\varphi$ with their approximations.}
\label{fig:reconstruc:Ejem}
\end{figure}
\begin{remark} If we consider any discretization of \eqref{B_Dfunc01}  on a given mesh, one has to solve a  system  like 
 $$
 A\vec{y}=\vec g.
 $$
 Obviously, the system depends strongly on the choice of the mesh.
  
 We notice  that the matrix $A$ may not be invertible. 
 While it is difficult to give a general criterion for the invertibility of  $A$ in terms of the mesh, Theorem \ref{B_Trec} guarantees
 that the matrix $A$ is indeed invertible when the non-regular mesh described in \eqref{nonRegularMesh} is used for the discretization of system \eqref{B_Dfunc01}. 
\end{remark}
\smallskip

Let us now consider the example studied in \cite{Bio:FrenchGroetsch}. 
To this end, we define
\beq \label{B_funcFrench}
I(t)=\left\{ 
\begin{array}{ll}
0, &  t\in (0,t_{Delay}) ,\\ \\
\D I_{Max} \left[1+ 
\left(\frac{K_I}{t-t_{Delay}}\right)^{n_I} \right]^{-1}, & t>t_{Delay}
\end{array}
\right.
\eeq 
with $t_{Delay} = 30 [ms]$, $n_I \simeq 2.2$, $I_{Max} = 150 [pA]$ and 
$K_I \simeq 100 [ms]$. The current given in  \eqref{B_funcFrench} is a sigmoidal function 
with short delay (Figure \ref{fig:reconstruc:01}{\bf B}), which is similar to the profiles encountered
in some practical situations (see e.g., \cite{Bio:Chen}, \cite{Bio:Flannery} or 
\cite{Bio:Koutalos}).
\begin{figure}[h] 
\begin{center}
\begin{minipage}[c]{7.5cm} 
\centering 
 {\bf A}: approximation of $\rho$.
\end{minipage}
\begin{minipage}[c]{6.5cm} 
\centering 
{\bf B}: current $I(t)$ as defined by \eqref{B_funcFrench}
\end{minipage}
\end{center}
\begin{minipage}[l]{7cm} 
\includegraphics[width=1.\textwidth]{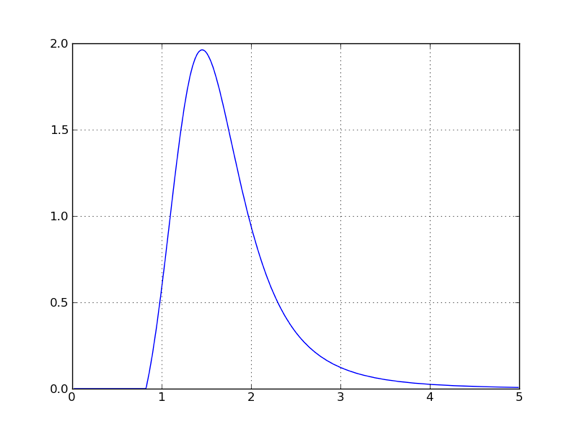} 
\end{minipage}
\begin{minipage}[r]{7cm} 
\includegraphics[width=1.\textwidth]{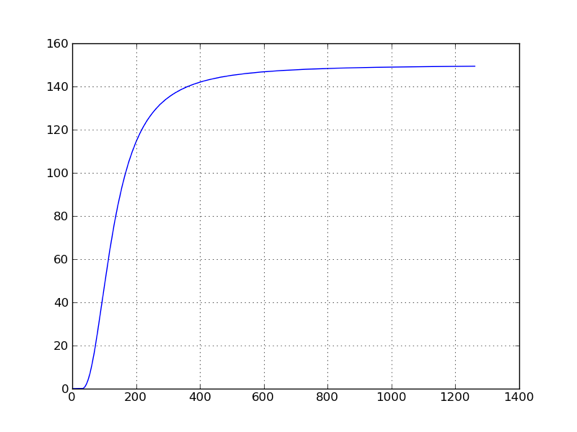} 
\end{minipage} 
\caption{Approximation of the function $\rho(x)$ with current $I(t)$ as defined by \eqref{B_funcFrench}}
\label{fig:reconstruc:01}
\end{figure}

The numerical solution corresponding to these data is shown in 
Figure~\ref{fig:reconstruc:01}{\bf A}. It should be noted that the numerical solution given here is 
perfectly consistent with those obtained in \cite{Bio:Flannery}. 

\section {Polynomial approximation of  Hill's function}\label{Int:PolinomialSection}
\
\par

In this section we consider the same inverse problem with another approximation of the kernel in \eqref{ABC}, for which we keep 
the function $c$ and replace Hill's function $F$ by a polynomial approximation around $c_0$. 
 
More precisely, let $P_m$ be the standard Taylor polynomial expansion of degree $m$ 
of \eqref{Int:Dhillfunc} around $c_0$; that is, $P_m\in \R [X]$, deg$(P_m )\le m$ and 
\beq
F(x)=P_m(x-c_0)+O(|x-c_0|^{m+1}).
\eeq
A new approximation for the kernel is defined by 
\beq\label{B_DkerAprox02}
PK_m(t,x)=P_m(c(t,x)-c_0),
\eeq
where $c(t,x)$ is the solution of \eqref{Int:Emodel2}, given by 
\beq
c(t,x)=c_0-c_0  \left( \frac{2}{L} \right)^{\frac{1}{2}} \left(\sum_{k=0}^{+\infty}\frac{e^{-\mu_k^2Dt}}{\mu_k}\psi_k(x)\right),
\eeq
with 
\beq\label{B_Emu:k}
\mu_k=\frac{2k+1}{2L}\pi,
\eeq  
and
\beq
\psi_k(x)=\left(\frac{2}{L}\right)^{1/2}\sin(\mu_k x).
\eeq  
Besides,  we define the total current associated with this polynomial 
approximation as follows:
\beq\label{B_DpoliCurrent}
PI_m[\rho](t)=\int_0^L\rho(x)PK_m(t,x)dx \quad\forall t>0.
\eeq

Next, we present our main result regarding the operator $PI_M$;
it asserts that the identifiability for the operator $PI_m$ holds when $m\le 8$.  
\begin{theorem} \label{B_Tiden03} Let $m\leq  8$ be a given integer. 
Then
$$
\textrm{Ker }PI_m=\{0\},
$$
where  
$$
\textrm{Ker }PI_m=\left\{f\in L^2(0,L)\Big|\;PI_m[f](t)=0\ \forall t>0 \right\}.
$$
\end{theorem}
 
\subsection{Proof of Theorem \ref{B_Tiden03}.}
Let us start noting that $\{\psi_k\}_{k\geq 0}$ is an orthonormal basis in 
$L^2(0,L)$. Thus, for any $f\in L^2(0,L),$ we can write 
\beq\label{B_EespecPhi}
f(x) = \sum_{k\geq0}<f,\psi_k>\psi_k(x) \qquad\textrm{ in } \  L^2(0,L),
\eeq
where 
$$
<f,\psi_k>=\int_0^Lf(x)\psi_k(x)dx.
$$

 We write 
$$
P_m(z) = \alpha_0 + \alpha_1z + \cdot\cdot\cdot + \alpha_mz^m,
$$
where $\alpha_j\in\R$, for all $j=0,1,...,m$, and introduce 
the set 
\beq
\Lambda _m = \left\{\sum_{j=1}^k \mu^2_{n_j}\;\Big| n_j\geq 0,\;\,1\leq k\leq m  \right\}.
\eeq

Let $\varepsilon>0$ be a given positive constant. For any $\rho \in L^2(0,L)$ and
for all $s\geq0$, we have 
\bea
PI_m[\rho](\varepsilon+s)&=&
\D\alpha_0\int_0^L\rho(x)dx  \nonumber \\ &+ &
\sum_{j=1}^m\alpha_j(- c_0\sqrt{\frac{2}{L}} )^j \int_0^L\left(\sum_{k\geq 0}  
\frac{e^{-\D\mu^2_{k}D(\varepsilon+s)}}{\mu_{k}}\psi_{k}(x)\right)^j\rho(x)dx\nonumber\\ 
&=&\D\alpha_0\int_0^L\rho(x)dx\nonumber\\ 
&+&\D\sum_{j=1}^m\alpha_j(-c_0\sqrt{\frac{2}{L}}   )^j \sum_{k_1,...,k_j\geq 0} e^{-\D\sum_{p=1}^j\mu^2_{k_p}D(s+ \varepsilon ) }
\int_0^L (\Pi_{p=1}^j   \frac{\psi_{k_p}(x)}{\mu _{k_p}} ) \rho(x)dx. \nonumber
\eea
Note that the convergence of the last series is fully justified, as for any $j\in \{ 1,...,m\}$ and any $s\ge 0$, we have
\begin{eqnarray*}
\sum_{k_1,...,k_j\ge 0} e^{-\D\sum_{p=1}^j\mu^2_{k_p}D(s+ \varepsilon ) }
\left\vert \int_0^L  ( \Pi_{p=1}^j  \frac{\psi_{k_p}(x)}{\mu _{k_p}} ) \rho(x)dx \right\vert  
&\le& || \rho ||_{L^1(0,L)} \left( \sum_{k\ge 0} \sqrt{ \frac{2}{L} } \frac{e^{-\mu_k ^2 D  \varepsilon }}{\mu _k} \right)^j\\
&<&\infty 
\end{eqnarray*}

Therefore, there is a family $\{a_\lambda(\varepsilon,\rho)\}_{\lambda\in\Lambda_m}$
such that 
\begin{equation}
\label{sum}
\sum_{\lambda \in \Lambda _m}  | a_\lambda (\varepsilon , \rho ) | <\infty
\end{equation}
 and 
\beq\label{B_Da:lambda}
\begin{array}{ccl}
PI_m[\rho](\varepsilon+s)
&=&\D \alpha_0\int_0^L\rho(x)dx + 
\sum_{\lambda \in \Lambda _m}a_\lambda(\varepsilon,\rho) e^{-\lambda Ds } \quad\forall s\geq0.
\end{array}
\eeq

\begin{lemma}\label{B_Lemma08} Let $\rho\in L^2(0,L)$ be a given function, such that 
\beq\label{B_L08E01}
PI_m[\rho](t) = 0 \quad\forall t>0.
\eeq
Then
\beq\label{B_L08E02}
\int_0^L\rho(x)dx = 0
\eeq
and 
\beq\label{B_L08E03}
a_\lambda(\varepsilon,\rho) = 0 \quad\forall\lambda\in\Lambda_m.
\eeq
\end{lemma}

\smallskip\
\begin{proof}
First,  the series in  \eqref{B_Da:lambda} is uniformly  convergent for $s\ge 0$, by \eqref{sum}. 

\smallskip
Define $\{\lambda_k\}_{k\geq 1}$ as
\bea\label{B_def:Lambda01}
\lambda_{1} = \D\min\Big\{\lambda\in \Lambda _m\Big\},\qquad
\lambda_{k+1} = \D\min\Big\{\lambda\in \Lambda _m\setminus\{\lambda_1,...,\lambda_k\}\Big\},
\eea
and note that  this defines an increasing sequence $0<\lambda_1<\lambda_2<\cdots$

\smallskip
Using definition \ref{B_def:Lambda01}, 
we can rewrite  \eqref{B_Da:lambda} as 
\beq\label{B_L08E04}
PI_m[\rho](\varepsilon+s) =
\D \alpha_0\int_0^L\rho(x)dx + 
\sum_{k\geq 1}a_{\lambda_k}(\varepsilon,\rho) 
e^{-\lambda_k D\varepsilon }e^{-\lambda_k Ds } \quad\forall s\geq 0.\qquad 
\eeq
Set  
$$
S_j(s) = \sum_{k\geq j}a_{\lambda_k}(\varepsilon,\rho) 
e^{-\lambda_k D\varepsilon }e^{-\lambda_k Ds } \quad\forall s\geq 0.
$$

Then we have that
\begin{equation}
\label{EDF}
|S_j(s)| \leq 
e^{-\lambda_j Ds}\left(\sum_{k\geq j}\left|a_{\lambda_k}(\varepsilon,\rho)
\right| e^{-\lambda_k D\varepsilon }\right).
\end{equation}
Plugging  \eqref{B_L08E04} into  \eqref{B_L08E01}, it follows that 
\beq\label{B_L08E06}
\alpha_0\int_0^L\rho(x)dx+S_1(s)=0 \quad\forall s>0.
\eeq
Passing to the limit as $s \to+\infty$ in \eqref{B_L08E06}, and noting that 
$\alpha_0 = F(c_0)\neq 0$ and that $S_1(s)\to 0$ by \eqref{EDF},  we obtain \eqref{B_L08E02}. 

\smallskip
The proof of  \eqref{B_L08E03} is done by induction on $j\ge 1$.  

\smallskip
{\bf Case j=1.} Plugging \eqref{B_L08E02} in \eqref{B_L08E06} and multiplying 
by $e^{\lambda_1Ds},$  we obtain
\beq\label{B_L08E07}
a_{\lambda_1}(\varepsilon,\rho)e^{-\lambda_1D\varepsilon} + 
e^{\lambda_1Ds}S_2(s)=0 \quad\forall s>0.
\eeq
But, using \eqref{EDF}, we also have that 
$$
|e^{\lambda_1Ds}S_2(s)|\leq 
C e^{-(\lambda_2-\lambda_1)Ds}.
$$
Thus, noting that $\lambda_1<\lambda_2$ and passing to the limit as $s \to \infty $ in \eqref{B_L08E07}, we obtain
$$
a_{\lambda_1}(\varepsilon, \rho) = 0.
$$

\smallskip
{\bf Case j=n+1.} Assume that 
\beq\label{B_L08E08}
a_{\lambda_j}(\varepsilon,\rho)=0 \quad\forall j=\{1,...,n\} .
\eeq
Plugging \eqref{B_L08E08} and \eqref{B_L08E02} in \eqref{B_L08E01}, we infer that
\beq\label{B_L08E09}
a_{\lambda_{n+1}}(\varepsilon,\rho)e^{-\lambda_{n+1}D\varepsilon} + 
e^{\lambda_{n+1}Ds}S_{n+2}(s)=0, \quad\forall s>0.
\eeq
On the other hand, using \eqref{EDF}, we have that 
$$
|e^{\lambda_{n+1}Ds}S_{n+2}(s)|\leq 
C e^{-(\lambda_{n+2}-\lambda_{n+1})Ds}.
$$
Thus, noting that $\lambda_{n+1}<\lambda_{n+2}$ and passing to the limit as 
$s \to \infty$ in \eqref{B_L08E09}, we infer that
$$
a_{\lambda_{n+1}}(\varepsilon, \rho) = 0.
$$
This yields \eqref{B_L08E03}.  This complete the proof of Lemma \ref{B_Lemma08}.
\end{proof}

\begin{lemma}\label{B_Lemma09} Let $\{\mu_k\}_{k\geq 0}$ be the sequence
defined in \eqref{B_Emu:k}. Assume that
 \beq\label{B_L09E01}
 \mu^2_{n_1}+\cdot\cdot\cdot+\mu^2_{n_k} = \mu_n^2.
 \eeq
for some $k\ge 1$ and $n,n_1,...,n_k\ge 0$. Then
\beq\label{B_L09E02}
k=1\ \textrm{mod } 8.
\eeq
\end{lemma}
\begin{proof}
We have 
$$
\mu^2_n = \frac{\pi^2}{4L^2}(4\varphi(n)+1),
$$
where $\varphi(n) = n^2+n.$ Thus, substituting this expression of $\mu _{n_i} ^2$ in \eqref{B_L09E01} yields
$$
k+4\sum_{i=1}^k\varphi(n_i) = 1+4\varphi(n).
$$
Noticing that $\varphi(n)$ is an even number for all $n\in\N,$ 
we obtain \eqref{B_L09E02}.
\end{proof}

\smallskip

 \begin{proof} [\bf Proof of the Theorem \ref{B_Tiden03}] 

\smallskip\
Let $\rho \in L^2(0,L)$ be a given function such that \eqref{B_L08E01}  holds. From Lemma~\ref{B_Lemma08} 
we infer that  \eqref{B_L08E03} holds.  If $m\leq 8$, using Lemma \ref{B_Lemma09} 
we have that for all $n\geq 0$, all $k\in \{ 2,...,m\}$, and all  $n_1,...,n_k\ge 0$,  
$$
\mu _{n_1}^2+ \cdots + \mu _{n_k}^2 \ne \mu _n^2. 
$$
Then, with $\lambda = \mu_n^2\in\Lambda_m$, we obtain
$$
a_{\mu _n^2}(\varepsilon,\rho) = e^{-\mu_n^2D\varepsilon}\alpha _1 (-c_0\sqrt{\frac{2}{L}})
\int_0^L\frac{\psi_n(x)}{\mu _n} \rho(x)dx.
$$
Since  $\alpha _1 = F'(c_0)\neq 0$ ($F$ being increasing), we infer that 
$$
<\psi_n,\rho>=0,\;\;\forall n\geq 0,
$$
and hence, with \eqref{B_EespecPhi}, that 
$$
\rho = 0.
$$
This completes the proof of Theorem~\ref{B_Tiden03}. 
 \end{proof}

\section{Acknowledgements}
The authors were partially supported by Basal-CMM project. The third author was partially 
supported by Fondecyt Grant N$^\circ$111102  and the fourth author was  partially supported by the 
\lq\lq Agence Nationale de la Recherche\rq\rq, Project CISIFS, Grant ANR-09-BLAN-0213-02.
The authors would like to thanks the organizers of the congress {\sl \lq\lq Partial Differential 
Equations, Optimal Design and Numerics\rq\rq}, held at Centro de Ciencias de Benasque Pedro Pascual, Benasque (Spain), and to the Basque Center for Applied Mathematics-BCAM, 
Bilbao (Spain), where part of this work was developed.

\end{document}